\documentclass[10pt]{amsart}
\usepackage{graphicx,amssymb,amsmath,amsfonts,amscd,amsthm,hyperref}
\usepackage[utf8x]{inputenc}

\newtheorem{thm}{Theorem}[section]
\newtheorem{cor}[thm]{Corollary}
\newtheorem{lem}[thm]{Lemma}
\newtheorem{prop}[thm]{Proposition}

\newtheorem{rem}[thm]{Remark}

\newcommand{\Hom}{\mathrm{Hom}}
\newcommand{\Fq}{\mathbb F_q}

\newcommand{\QQ}{\bar{\mathbb Q}_\ell}
\newcommand{\R}{\mathrm R}

\newcommand{\Gal}{\mathrm {Gal}}
\newcommand{\AAA}{\mathbb A}
\newcommand{\PP}{\mathbb P}
\newcommand{\FF}{\mathcal F}
\newcommand{\GG}{\mathbb G}
\newcommand{\GGG}{\mathcal G}

\newcommand{\HH}{\mathrm H}
\newcommand{\HHH}{{\mathcal H}}
\newcommand{\LL}{{\mathcal L}}

\newcommand{\Tr}{\mathrm {Tr}}

\newcommand{\Dbc}{{\mathcal D}^b_c}

\newcommand{\Swan}{\mathrm{Swan}}

\newcommand{\Sh}{\mathcal Sh}

\newcommand{\PPP}{\mathcal P}

\newcommand{\FT}{\mathrm{FT}}
\newcommand{\RR}{\mathcal R}
\newcommand{\Gmk}{\GG_{m,\bar k}}
\newcommand{\istar}{\iota^\star}

\allowdisplaybreaks

\title{Local convolution of $\ell$-adic sheaves on the torus}
\author{Antonio Rojas-Le\'on}
\address{Departamanto de \'Algebra,
Universidad de Sevilla, Apdo 1160, 41080 Sevilla, Spain}
\address{E-mail: arojas@us.es}
\thanks{Partially supported by P08-FQM-03894 (Junta de Andaluc\'{\i}a), MTM2010-19298 and FEDER}

\begin{document}

\renewcommand{\thefootnote}{}

\begin{abstract}
For $K$ and $L$ two $\ell$-adic perverse sheaves on the one-dimensional torus $\Gmk$ over the algebraic closure of a finite field, we show that the local monodromies of their convolution $K\ast L$ at its points of non-smoothness is completely determined by the local monodromies of $K$ and $L$. We define local convolution bi-exact functors $\rho_{(s,t)}^{(u)}$ for every $s,t,u\in\PP^1_{\bar k}$ that map continuous $\ell$-adic representations of the inertia groups at $s$ and $t$ to a representation of the inertia group at $u$, and show that the local monodromy of $K\ast L$ at $u$ is the direct sum of the $\rho_{(s,t)}^{(u)}$ applied to the local monodromies of $K$ and $L$. This generalizes a previous result of N. Katz for the case where $K$ and $L$ are smooth, tame at $0$ and totally wild at infinity.  
\end{abstract}

\maketitle

\section{Introduction}

Let $k=\Fq$ be a finite field of characteristic $p$, $\bar k$ a fixed algebraic closure and $G$ a one-dimensional smooth affine group scheme over $\bar k$ (so $G$ is either the affine line $\AAA^1_{\bar k}$ or the torus $\Gmk$). Fix a prime number $\ell\neq p$, and let $\Dbc(G,\QQ)$ be the derived category of constructible $\ell$-adic sheaves on $G$. The convolution is a triangulated bifunctor $\Dbc(G,\QQ)\times\Dbc(G,\QQ)\to\Dbc(G,\QQ)$ given by \cite[8.1.8]{katz1990esa}
$$
K\ast_!L=\R\sigma_!(K\boxtimes L)
$$
where $\sigma:G\times G\to G$ is the group operation map. 

If $G$ is obtained from a group scheme $G_0$ over $k$ by extension of scalars, then $K\ast_! L$ is naturally defined over $k$ for every $K,L\in\Dbc(G_0,\QQ)$. Moreover, on the level of Frobenius traces the operation corresponds to the usual convolution in finite abelian groups: for every $t\in k$, if $K(t)$ denotes the trace of a geometric Frobenius element at $t$ acting on the stalk of $K$ at a geometric point over $t$ (and similarly for $L$) then
$$
(K\ast_! L)(t)=\sum_{uv=t}K(u)L(v).
$$
This is an easy consequence of Grothendieck's trace formula.

There is another variant ($\star$-convolution) where one takes direct image without supports: $K\ast_\star L=\R\sigma_*(K\boxtimes L)$. If we ignore objects belonging to a certain subcategory of $\Dbc(G,\QQ)$ (those whose cohomology sheaves are of Artin-Schreier type in the additive case, and of Kummer type in the multiplicative case) both operations coincide, and they preserve the subcategory of perverse objects \cite[2.6]{katz1996rls},\cite[Proposition 3.6.4]{gabber1996faisceaux}.

If $G=\AAA^1_{\bar k}$, for any non-trivial additive character $\psi:k\to\QQ^\star$ the Fourier transform with respect to $\psi$ $\FT^\psi:\Dbc(\AAA^1_{\bar k},\QQ)\to\Dbc(\AAA^1_{\bar k},\QQ)$ is an auto-equivalence of categories and interchanges convolution and tensor product \cite[Corollaire 9.6]{brylinski1986transformations}. For a perverse object $K\in\Dbc(\AAA^1_{\bar k},\QQ)$, Laumon's local Fourier transform describes the local monodromies of $\FT^\psi(K)$ at its points of non-smoothness in terms of those of $K$. In particular, if $K,L\in\Dbc(\AAA^1_{\bar k},\QQ)$ are perverse, we can completely determine the local monodromies of $K\ast_! L$ in terms of those of $K$ and $L$ (modulo Artin-Schreier objects), see for instance \cite[2.7]{laumon1987transformation}.

In the multiplicative case things are more complicated, since there is no algebraic multiplicative analog of the $\ell$-adic Fourier transform. For a particular class of objects $K\in\Dbc(\Gmk,\QQ)$ (those which are smooth on $\Gmk$, tamely ramified at $0$ and totally wild at infinity) N. Katz and O. Gabber proved \cite[Chapters 6 and 7]{katz1988gauss} that the local monodromies at $0$ and $\infty$ of $K\ast_!L$ are completely determined by those of $K$ and $L$. In this article we will generalize this to arbitrary objects $K,L\in\Dbc(\Gmk,\QQ)$.

More precisely, let $X\subseteq \PP^1_{\bar k}\times\PP^1_{\bar k}\times\PP^1_{\bar k}$ be the Zariski closure of the set $\{(x,y,z)\in\Gmk^3|z=xy\}\subseteq\Gmk^3$. For every $t\in\PP^1(\bar k)$ let $I_t\subseteq\Gal(\bar k(t)^{sep}/\bar k(t))$ denote the inertia group at $t$, and let $\RR_t$ (respectively $\RR_t^w$) be the abelian category of finite dimensional continuous $\ell$-adic representations of $I_t$ (resp. the totally wild ones). We will construct bi-exact functors $\rho_{(s,t)}^{(u)}:\RR^\star_s\times\RR^\star_t\to\RR^\star_u$ for every $(s,t,u)\in X$ (where $\RR^\star_t=\RR^w_t$ for $t=0$ or $t=\infty$ and $\RR^\star_t=\RR_t$ otherwise) such that, for any semisimple perverse objects $K,L\in\Dbc(\Gmk,\QQ)$, the local monodromy of $K\ast_! L$ at $u\in\PP^1(\bar k)$ is given by
$$
\bigoplus_{(s,t,u)\in X}\rho_{(s,t)}^{(u)}(K_{(s)},L_{(t)})
$$
where $K_{(s)}$ and $L_{(t)}$ denote the local monodromies of $K$ at $s$ and of $L$ at $t$ respectively, in a sense that will be made precise later.

We now describe briefly the structure of the article. In the first section we review the main results about convolution and Fourier transform that we will make use of. In the second section we prove a conjecture by Katz which gives a precise description of the monodromy at infinity of the convolution of two objects which are smooth on $\Gmk$, tame at $0$ and totally wild at infinity. This result will be used later to deduce some properties of the functors $\rho_{(s,t)}^{(u)}$. The next two sections make up the core of the article, and in them we define the functors $\rho_{(s,t)}^{(u)}$ (for $u=\infty$ or $0$ and for $u\in\bar k^\star$ respectively) and prove the main theorems \ref{infinity} and \ref{finite}. In the last section we discuss what can be said about the tame part of the monodromy at $0$ and infinity.

\section{Review of convolution on $\GG_m$ and local Fourier transform}

In this section we will review the main definitions and results to be used throughout this article. Let $k$ be a finite field, $\bar k$ a fixed algebraic closure and $\Gmk$ the $1$-dimensional multiplicative torus over $\bar k$. Fix a prime $\ell$ different from the characteristic of $k$. Let $\Sh(\Gmk,\QQ)$ be the abelian category of constructible $\ell$-adic sheaves on $\Gmk$, and $\Dbc(\Gmk,\QQ)$ the corresponding derived category.

Given two objects $K,L\in\Dbc(\Gmk,\QQ)$, their \emph{convolution} is the object $$K\ast_!L=\R\mu_!(K\boxtimes L)\in\Dbc(\Gmk,\QQ),$$ 
where $\mu:\Gmk\times\Gmk\to\Gmk$ is the multiplication map. There is also a ``without supports'' variant
$$
K\ast_\star L=\R\mu_\star(K\boxtimes L).
$$
They are both associative and commutative triangulated bifunctors. The two types of convolution are interchanged by duality \cite[2.5]{katz1996rls}.

Let $S\subseteq\Dbc(\Gmk,\QQ)$ be the full subcategory consisting of objects whose cohomology sheaves are succesive extensions of Kummer sheaves $\LL_\chi$ for different finite order characters $\chi:\bar k^\star\to\QQ$ \cite[1.4-1.8]{deligne569application}. Then $S$ is a thick subcategory and an ideal under both types of convolution \cite[3.6.2, 3.6.4]{gabber1996faisceaux}, so each of them descends to a bifunctor on the quotient category $\Dbc(\Gmk,\QQ)/S$. Moreover, for every $K,L\in\Dbc(\Gmk,\QQ)$ the mapping cone of the natural ``forget supports'' map $K\ast_!L\to K\ast_\star L$ is in $S$ \cite[Proposition 3.6.4]{gabber1996faisceaux}. In particular both types of convolution define the same operation on $\Dbc(\Gmk,\QQ)/S$, which we will simply call convolution and denote by $K\ast L$.

Let $\PPP$ be the subcategory of $\Dbc(\Gmk,\QQ)/S$ consisting of (the objects isomorphic to) perverse objects, that is, objects $K$ such that $\HHH^i(K)=0$ for $i\neq -1,0$, $\HHH^0(K)$ is punctually supported and $\HHH^{-1}(K)$ does not have punctual sections. It is an abelian category, and by \cite[Proposition 3.6.4(iii)]{gabber1996faisceaux}, the convolution (in $\Dbc(\Gmk,\QQ)/S$) of two objects in $\PPP$ is in $\PPP$. Therefore, convolution defines a bi-exact associative and commutative functor $\PPP\times\PPP\to\PPP$. 

Let $K\in\Dbc(\Gmk,\QQ)$ be a perverse object, and consider the action of the monodromy group $I_0$ at $0$ on the generic stalk $V$ of $\HHH^{-1}(K)$. By \cite[Proposition 1.1, Lemma 1.8]{katz1988gauss}, there is a canonical decomposition $V=V^t\oplus V^w$ into its tame part (on which the wild inertia group $P_0\subseteq I_0$ acts trivially) and its wild part (on which $P_0$ acts with no invariants). Then $V^w$ does only depend on the class of $K$ in $\PPP$: If $K\cong L$ in $\PPP$, $K$ and $L$ can be joined by a chain of maps $K_n\to L_n$ whose mapping cones $M_n$ are in $S$. We have then exact sequences
$$ 
\HHH^{-1}(M_n)\to\HHH^{-1}(K_n)\to\HHH^{-1}(L_n)\to\HHH^0(M_n).
$$
But as representations of $I_0$ both $\HHH^{-1}(M_n)$ and $\HHH^{0}(M_n)$ are tame, so there is an isomorphism between the wild part of the monodromies at $0$ of $\HHH^{-1}(K_n)$ and $\HHH^{-1}(L_n)$.

For any $K\in\PPP$, we denote by $K_{(0)}^w$ the wild part $V^w$ of the action of $I_0$ on $\HHH^{-1}(K)$. Similarly, the wild part of the action of the inertia group $I_\infty$ at infinity on $\HHH^{-1}(K)$ only depends on the class of $K$ in $\PPP$, we will denote it by $K_{(\infty)}^w$.

We will say that an object $K\in\PPP$ is \emph{semisimple} if it is isomorphic in $\PPP$ to a semisimple perverse sheaf. By \cite[paragraph after Lemma 3.3]{katz2010mellin}, the convolution of two semisimple objects is semisimple. For $K\in\Dbc(\Gmk,\QQ)$ a semisimple perverse object and $t\in \bar k^\star$, consider the specialization map $K_t\to K_{\bar\eta}$, where $\bar\eta$ is a geometric generic point of $\Gmk$, $K_t$ and $K_{\bar\eta}$ are the stalks of $K$ at $t$ and $\bar\eta$ respectively and $I_t$ acts trivially on $K_t$. Let $C_t$ be its mapping cone viewed as an object of the derived category of representations of the inertia group $I_t$ at $t$. Since $\HHH^{-1}(K)$ does not have punctual sections and $\HHH^0(K)$ is punctual, the cohomology of $C_t$ is concentrated in degree $-1$. We define $K_{(t)}=\HHH^{-1}(C_t)$. We have an exact sequence of $I_t$-representations
$$
0\to \HHH^{-1}(K)_{\bar\eta}/\HHH^{-1}(K)_t \to K_{(t)} \to\HHH^0(K)_t\to 0
$$
where $\HHH^{-1}(K)_t$ is the stalk of $\HHH^{-1}(K)$ at $t$, which coincides with the $I_t$-invariant subspace of $\HHH^{-1}(K)_{\bar\eta}$ since $K$ is semisimple (so $\HHH^{-1}(K)$ is a direct sum of middle extensions).

It is clear that if $K$ and $L$ are semisimple perverse sheaves which are isomorphic in $\PPP$ then $K_{(t)}\cong L_{(t)}$ as representations of $I_t$ for every $t\in\bar k^\star$ (since $M_{(t)}=0$ if $M$ is an extension of Kummer objects). We can then define $K_{(t)}:=L_{(t)}$ for every semisimple $K\in\PPP$, where $L$ is any semisimple perverse sheaf isomorphic to $K$ in $\PPP$.  

Fix a non-trivial additive character $\psi:k\to\QQ^\star$. The Fourier transform $\FT^\psi$ with respect to $\psi$ is an autoequivalence of categories $\Dbc(\AAA^1_{\bar k},\QQ)\to\Dbc(\AAA^1_{\bar k},\QQ)$, whose inverse is the Fourier transform $\FT^{\bar\psi}$ with respect to the conjugate character. Let $j:\Gmk\to\AAA^1_{\bar k}$ be the inclusion. Given an object $K\in\Dbc(\Gmk,\QQ)$, we set $\FT^\psi(K):=j^\star\FT^\psi(j_!K)$. If $K\in S$ then $\FT^\psi(K)\in S$, since the Fourier transform of $j_!\QQ[1]$ restricted to $\Gmk$ is again constant, and the Fourier transform of an object of the form $j_!\LL_\chi[1]$ is an object of the same form \cite[Proposition 1.4.3.2]{laumon1987transformation}. Therefore $\FT^\psi$ descends to a functor in the quotient category $\Dbc(\Gmk,\QQ)/S$. In particular, since $\FT^\psi$ takes perverse objects to perverse objects \cite[Corollaire 2.1.5]{katz1985transformation}, it defines a functor $\FT^\psi:\PPP\to\PPP$. It is an auto-equivalence of categories with inverse $\FT^{\bar\psi}$.

Laumon's results relate the local monodromies of an object $K\in\PPP$ and its Fourier transform. The following theorem summarizes the results in our notation. Throughout this article we will refer to this result as Local Fourier transform theory (LFTT).

\begin{thm}\label{LFTT}\cite{laumon1987transformation,katz1988travaux}
 For every $t\in\PP^1(\bar k)$ let $\RR_t$ (respectively $\RR_t^w$) be the category of representations of the inertia group $I_t\subseteq\mathrm{Gal}(\bar k(x)^{sep}/\bar k(x))$ (resp. the category of totally wild representations of $I_t$). There exist exact functors $\FT_{(0,\infty)}:\RR_0\to\RR_\infty$, $\FT_{(\infty,\infty)}:\RR_\infty\to\RR_\infty$, $\FT_{(\infty,0)}:\RR_\infty\to\RR_0$, $\FT_{(t,\infty)}:\RR_t\to\RR_\infty$ and $\FT_{(\infty,t)}:\RR_\infty\to\RR_t$ for $t\in\bar k^\star$ such that for every semisimple $K\in\PPP$ we have
\begin{enumerate}
 \item $(\FT^\psi K)^w_{(\infty)}\cong\FT_{(0,\infty)}K^w_{(0)}\oplus\FT_{(\infty,\infty)}K^w_{(\infty)}\oplus\bigoplus_{t\in\bar k^\star}\FT_{(t,\infty)}K_{(t)}.$
 \item $(\FT^\psi K)^w_{(0)}\cong\FT_{(\infty,0)}K^w_{(\infty)}.$
 \item $(\FT^\psi K)_{(t)}\cong\FT_{(\infty,t)}K^w_{(\infty)}$ for $t\in\bar k^\star$.
\end{enumerate}
Moreover, these functors have the following properties:
\begin{enumerate}
 \item If $\FF\in\RR_0$ has a single slope $a\geq 0$ and dimension $n$, $\FT_{(0,\infty)}\FF$ has a single slope $\frac{a}{a+1}$ and dimension $(a+1)n$.
 \item If $\FF\in\RR_\infty$ has a single slope $a\geq 0$ and dimension $n$, $\FT_{(\infty,\infty)}\FF$ is $0$ if $a\leq 1$, and has a single slope $\frac{a}{a-1}$ and dimension $(a-1)n$ if $a>1$.
 \item If $\FF\in\RR_\infty$ has a single slope $a\geq 0$ and dimension $n$, $\FT_{(\infty,0)}\FF$ is $0$ if $a\geq 1$, and has a single slope $\frac{a}{1-a}$ and dimension $(1-a)n$ if $a<1$.
 \item For $t\in\bar k^\star$, if we identify $\RR_0$ and $\RR_t$ via the translation map $x\mapsto x+t$, we have $\FT_{(t,\infty)}\FF\cong\LL_{\psi_t}\otimes\FT_{(0,\infty)}\FF$ and $\FT_{(\infty,t)}\FF\cong\FT_{(\infty,0)}(\FF\otimes\LL_{\psi_{-t}})$ where $\LL_{\psi_t}\in\RR_\infty$ is the representation given by the Artin-Schreier sheaf associated to the character $\psi_t:k_0\to k_0$, $x\mapsto \psi(\Tr_{k_0/k}(tx))$ and $k_0=k(t)$. In particular, $\FT_{(t,\infty)}\FF$ has slope $1$ for any $\FF\in\RR_t$, and $\FT_{(\infty,t)}\FF=0$ for any $t\in\bar k^\star$ if $\FF$ does not have $1$ among its slopes.
\end{enumerate}

\begin{proof}
 Since all properties are preserved by taking direct sums, we may assume that $K$ is a simple perverse sheaf. Then $K$ is either punctual on degree $0$, or an irreducible middle extension on degree $-1$. In the latter case, if $\HHH^{-1}(K)$ is not of Artin-Schreier type then it is a Fourier sheaf in the sense of \cite[8.2]{katz1988gauss}, and the theorem is just a rewrite of the results in \cite[Th\'eor\`eme 2.4.3]{laumon1987transformation},\cite[Theorems 8-13]{katz1988travaux} for the functors defined in \cite[2.4.2.3]{laumon1987transformation}. If $\HHH^{-1}(K)$ is an Artin-Schreier sheaf $\LL_{\psi_t}$ then the Fourier transform of $K$ is punctual supported on $t$. If $K$ is punctual, $K=\delta_t[0]$ for some $t\in\bar k^\star$, and its Fourier transform is $\LL_{\psi_t}[1]$. So it only ramains to check that $\FT_{(t,\infty)}$ of the trivial character is the character of $I_\infty$ induced by the Artin-Schreier sheaf $\LL_{\psi_t}$ and, reciprocally, that $\FT_{(\infty,t)}$ of the character $\LL_{\psi_t}$ is the trivial character. This is just a consequence of property (4) above and the fact that $\FT_{(0,\infty)}$ and $\FT_{(\infty,0)}$ take the trivial character to the trivial character \cite[Proposition 2.5.3.1]{laumon1987transformation}. 
\end{proof}

\end{thm}

\section{A conjecture of Katz}

We keep the notation and hypotheses of the previous section. Let $\mathcal C$ be the subcategory of $\PPP$ consisting of objects isomorphic to $\LL[1]$, where $\LL\in\Sh(\Gmk,\QQ)$ is smooth on $\GG_{m,\bar k}$, tamely ramified at $0$ and totally wildly ramified at infinity. By \cite[Theorem 5.1(1)]{katz1988gauss}, ${\mathcal C}$ is invariant under convolution. More precisely, if $\FF,\GGG\in\Sh(\Gmk,\QQ)$ are smooth, tame at $0$ and totally wild at infinity, then $\FF[1]\ast_!\GGG[1]\cong\FF[1]\ast_\star\GGG[1]\cong{\mathcal H}[1]$, where $\HHH\in\Sh(\Gmk,\QQ)$ has the same properties. By abuse of language, we will write $\FF\ast\GGG=\HHH$ in that case.

In \cite[7.6]{katz1988gauss}, N. Katz gives a conjectural formula for the slopes (and their multiplicities) of $\FF\ast\GGG$ at infinity in terms of the slopes and multiplicities of $\FF$ and $\GGG$ for $\FF[1],\GGG[1]\in{\mathcal C}$. One of its equivalent formulations is the following:

\begin{prop}\label{slopes}
 Let $K=\FF[1],L=\GGG[1]\in{\mathcal C}$ have single slopes $a>0$ and $b>0$ at infinity respectively. Then $K\ast L$ has a single slope $\frac{ab}{a+b}$ at infinity.
\end{prop}

We will prove the proposition via some technical lemmas.

\begin{lem}\label{lema1}
 Let $K,L,M\in\Dbc(\GG_{m,\bar k},\QQ)$.  Then 
$$
\R\Gamma_c(\GG_{m,\bar k},(K\ast_! L)\otimes M)\cong\R\Gamma_c(\GG_{m,\bar k},K\otimes((\istar L)\ast_! M))
$$
where $\iota:\Gmk\to\Gmk$ is the inversion map.
\end{lem}

\begin{proof}
 If $\mu:\GG_{m,\bar k}\times\GG_{m,\bar k}\to\GG_{m,\bar k}$ is the multiplication map, we have
\begin{align*}
&\R\Gamma_c(\GG_{m,\bar k},(K\ast_! L)\otimes M)=\R\Gamma_c(\GG_{m,\bar k},\R\mu_!(K\boxtimes L)\otimes M)=
\\
&=\R\Gamma_c(\GG_{m,\bar k},\R\mu_!((K\boxtimes L)\otimes\mu^\star M))
=\R\Gamma_c(\GG_{m,\bar k}\times\GG_{m,\bar k},(K\boxtimes L)\otimes\mu^\star M)
\end{align*}
by the projection formula. If $\pi_1,\pi_2:\GG_{m,\bar k}\times\GG_{m,\bar k}\to\GG_{m,\bar k}$ are the projections then 
$$
\R\Gamma_c(\GG_{m,\bar k}\times\GG_{m,\bar k},(K\boxtimes L)\otimes\mu^\star M)
=\R\Gamma_c(\GG_{m,\bar k}\times\GG_{m,\bar k},\pi_1^\star K\otimes \pi_2^\star L\otimes\mu^\star M).
$$
Consider the automorphism $\phi:\Gmk\times\Gmk\to\Gmk\times\Gmk$ given by $(x,y)\mapsto (xy,y^{-1})$. Then $\mu=\pi_1\circ\phi$, $\pi_1=\mu\circ\phi$ and $\iota\circ\pi_2=\pi_2\circ\phi$, where $\iota:\Gmk\to\Gmk$ is the inversion map. It follows that
\begin{align*}
&\R\Gamma_c(\GG_{m,\bar k}\times\GG_{m,\bar k},\pi_1^\star K\otimes \pi_2^\star L\otimes\mu^\star M)\cong
\\
&\cong\R\Gamma_c(\GG_{m,\bar k}\times\GG_{m,\bar k},\mu^\star K\otimes \pi_2^\star \istar L\otimes\pi_1^\star M)=\R\Gamma_c(\Gmk,K\otimes\R\mu_!((\istar L)\boxtimes M))=
\\
&=\R\Gamma_c(\Gmk,K\otimes((\istar L)\ast_! M)).
\end{align*}
\end{proof}

\begin{cor}\label{lema1cor}
  Let $K,L\in\Dbc(\GG_{m,\bar k},\QQ)$, $\psi:k\to\QQ^\star$ a non-trivial additive character and $\LL_{\psi}$ the corresponding Artin-Schreier sheaf. Then 
$$
\R\Gamma_c(\GG_{m,\bar k},(K\ast_! L)\otimes \LL_\psi)\cong\R\Gamma_c(\GG_{m,\bar k},K\otimes\FT^\psi L)[-1].
$$
\end{cor}

\begin{proof}
 This is immediate from the previous lemma and the formula \cite[Proposition 8.1.12]{katz1990esa}
$$
(\istar L)\ast_!\LL_\psi[1]\cong \FT^\psi L.
$$
\end{proof}

\begin{lem}\label{lema2}
 Let $K=\FF[1],L=\GGG[1]\in{\mathcal C}$ have single slopes $a>0$ and $b>0$ at infinity respectively. Suppose that $1<\frac{1}{a}+\frac{1}{b}<2$. Let $\psi:k\to\QQ^\star$ be a non-trivial character and $\LL_\psi$ the corresponding Artin-Schreier sheaf. Then the Swan conductor of $(K\ast L)\otimes\LL_\psi$ at infinity is equal to its rank. 
\end{lem}

\begin{proof}
Let $m,n$ be the ranks of $\FF$ and $\GGG$ respectively. According to \cite[Theorem 5.1(4)]{katz1988gauss} the rank of $\FF\ast\GGG$ (and therefore also of $(\FF\ast\GGG)\otimes\LL_\psi$) is $mn(a+b)$, so we need to show that the Swan conductor at infinity of $(\FF\ast\GGG)\otimes\LL_\psi$ is also $mn(a+b)$.

 Since $(\FF\ast\GGG)\otimes\LL_\psi[1]\in{\mathcal C}$ is smooth on $\GG_{m,\bar k}$ and tame at $0$, by the Ogg-Shafarevic formula we have $\Swan_\infty((\FF\ast\GGG)\otimes\LL_\psi)=-\chi(\GG_{m,\bar k},(\FF\ast\GGG)\otimes\LL_\psi)$. But by corollary \ref{lema1cor} the Euler characteristics of $(\FF\ast\GGG)\otimes\LL_\psi=\HHH^{-1}((K\ast_! L)\otimes\LL_\psi)$ and $K\otimes\FT^\psi L$ are equal.

Since $\frac{1}{a}+\frac{1}{b}<2$, at least one of $a,b$ is greater than $1$. By the commutativity of the convolution, we may assume without loss of generality that $b>1$. Then $\GGG$ has no Artin-Schreier components, so $\FT^\psi L=\HHH[1]$ for a middle extension sheaf $\HHH\in\Sh(\GG_{m,\bar k},\QQ)$. Since $\GGG$ has slope $b$ with multiplicity $n$ at infinity, LFTT tells us that $\HHH$ is smooth on $\GG_{m,\bar k}$, tame at $0$, and the wild part of its local monodromy at infinity has $\frac{b}{b-1}$ as its only slope, with multiplicity $n(b-1)$. On the other hand, the rank of $\HHH$ is minus the Euler characteristic of $\GGG\otimes\LL_{\psi}$. Since $\LL_\psi$ has slope $1$ at infinity and $b>1$, $\GGG\otimes\LL_{\psi}$ has a single slope $b$ at infinity and therefore its Euler characteristic is $-nb$ by the Ogg-Shafarevic formula. We conclude that $\HHH$ has rank $nb$, and in particular its tame part at infinity has dimension $nb-n(b-1)=n$.

Then $\FF\otimes\HHH$ is smooth on $\GG_{m,\bar k}$ and tame at $0$, so its Euler characteristic is minus its Swan conductor at infinity. Since $\FF$ has a single slope $a$ with multiplicity $m$ and $\HHH$ has slope $\frac{b}{b-1}>a$ with multiplicity $n(b-1)$ and slope $0$ with multiplicity $n$, we conclude that $\FF\otimes\HHH$ has slope $\frac{b}{b-1}$ with multiplicity $mn(b-1)$ and slope $a$ with multiplicity $mn$ \cite[Lemma 1.3]{katz1988gauss}. Its Swan conductor is then $mn(b-1)\frac{b}{b-1}+mna=mn(a+b)$.
\end{proof}

\begin{lem}\label{lema3}
 Let $K=\FF[1],L=\GGG[1]\in{\mathcal C}$ have single slopes $a>0$ and $b>0$ at infinity respectively. Suppose that $1<\frac{1}{a}+\frac{1}{b}<2$. Then all slopes of $K\ast L$ at infinity are $\leq 1$.
\end{lem}

\begin{proof}
 Let $M$ be the generic stalk of $\FF\ast\GGG$ as a representation of $P_\infty$, the wild inertia group at infinity. Then we have a decomposition \cite[Proposition 1.1]{katz1988gauss}
$$
M\cong M_{< 1}\oplus M_{=1}\oplus M_{> 1}
$$
where $M_{< 1}$ (respectively $M_{=1}$, $M_{> 1}$) has all slopes $< 1$ (resp. $=1$, $> 1$). Fix a non-trivial character $\psi:k\to\QQ^\star$. By \cite[Lemma 8.5.7]{katz1988gauss}, for every $P_\infty$-irreducible subspace $N$ of $M_{=1}$ there is exactly one $a\in\bar k^\star$ such that $N\otimes\LL_{\psi(bt)}$ has all slopes equal to $1$ for every $b\neq a$, where $\LL_{\psi(bt)}$ is the pull-back of the Artin-Schreier sheaf $\LL_\psi$ by the multiplication by $b$ map. Therefore, for all but finitely many $b\in \bar k$ the $P_\infty$-representation $M_{=1}\otimes\LL_{\psi(bt)}$ has all its slopes equal to $1$. Then for all such $b$
$$
M\otimes\LL_{\psi(bt)}=(M_{< 1}\otimes\LL_{\psi(bt)})\oplus(M_{=1}\otimes\LL_{\psi(bt)})\oplus(M_{> 1}\otimes\LL_{\psi(bt)})
$$
has all slopes $\geq 1$, and all of them equal to $1$ if and only if $M_{> 1}=0$ \cite[Lemma 1.3]{katz1988gauss}. In particular, the Swan conductor of $(\FF\ast\GGG)\otimes\LL_{\psi(bt)}$ at infinity is greater than or equal to its rank, with equality if and only if $M_{> 1}=0$. We conclude by lemma \ref{lema2} applied to the character $\psi(bt)$ (extending scalars to a finite extension of $k$ if necessary). 
\end{proof}

\begin{lem}\label{lema4}
 Let $K=\FF[1],L=\GGG[1]\in{\mathcal C}$ have single slopes $a>0$ and $b>0$ at infinity respectively. Then all slopes of $K\ast L$ at infinity are $\leq \frac{ab}{a+b}$.
\end{lem}

\begin{proof}
 Let $(r_n)_{n\geq 1}$ be a sequence of rational numbers such that
\begin{enumerate}
 \item $\frac{1}{2}\left(\frac{1}{a}+\frac{1}{b}\right)<r_n<\frac{1}{a}+\frac{1}{b}$ for every $n\geq 1$.
 \item $r_n$ has $p$-adic valuation $0$ for every $n\geq 1$.
 \item $r_n\to\frac{1}{a}+\frac{1}{b}$ as $n\to\infty$.
\end{enumerate}

If $r=\frac{m}{n}$ is a rational number with $m,n\geq 1$ relatively prime and prime to $p$, for every $\QQ$-sheaf $\HHH$ on $\GG_{m,\bar k}$ we denote by $[r]^\star\HHH$ the sheaf $[m]^\star[n]_\star\HHH$, where $[m]$ and $[n]:\Gmk\to\Gmk$ are the $m$-th and $n$-th power maps respectively. By \cite[1.13.1, 1.13.2]{katz1988gauss} for every $n\geq 1$ the sheaves $[r_n]^\star\FF$ and $[r_n]^\star\GGG$ have single slopes $r_na$ and $r_nb$ at infinity respectively, and they are still in $\mathcal C$ \cite[Lemma 5.0.1(5)]{katz1988gauss}. Since $1<\frac{1}{r_na}+\frac{1}{r_nb}<2$ by hypothesis, we can apply lemma \ref{lema3} to them, and we deduce that $[r_n]^\star\FF\ast[r_n]^\star\GGG$ has all slopes $\leq 1$. 

Now by \cite[Theorem 5.1(10,12)]{katz1988gauss} there is an injection $[r_n]^\star(\FF\ast\GGG)\hookrightarrow[r_n]^\star\FF\ast[r_n]^\star\GGG$, so $[r_n]^\star(\FF\ast\GGG)$ has all slopes $\leq 1$ and therefore $\FF\ast\GGG$ has all slopes $\leq r_n^{-1}$ \cite[1.13.1, 1.13.2]{katz1988gauss}. We conclude by taking $n\to\infty$.
\end{proof}

\begin{proof}[Proof of proposition \ref{slopes}]
 Let $m,n$ be the ranks of $\FF$ and $\GGG$ respectively. By \cite[Theorem 5.1(4,5)]{katz1988gauss} $\FF\ast\GGG$ has rank $mn(a+b)$ and Swan conductor $mnab$ at infinity. Since all $mn(a+b)$ slopes at infinity are $\leq \frac{ab}{a+b}$ by lemma \ref{lema4} and they add up to $mnab$, thay must all be equal to $\frac{ab}{a+b}$.
\end{proof}

\begin{rem}\emph{
 By \cite[Remark 7.22]{rl2010}, this result implies the following estimate for exponential sums associated to homothety invariant polynomials: Let $g\in k[x]$ be a polynomial of degree $d$ prime to $p$ which is not of the form $h(x^n)$ for any $n\geq 2$, and let $e$ a positive integer that divides $q-1$. Then for every $r\geq 1$ we have the estimate:
$$
\left|\sum_{x\in k_r^\star}\psi(\Tr_{k_r/k}(g(x^{\frac{q-1}{e}})))\right|\leq C_{d,r}(q-1)q^{\frac{r-1}{2}}
$$
where $k_r$ is the extension of $k$ of degree $r$ in $\bar k$ and 
$$
C_{d,r}=\frac{r}{d}\sum_{i=0}^{r-1}\binom{d+r-i-1}{r}\binom{r-1}{i}.
$$}
\end{rem}

\section{Local monodromy at infinity of a convolution}

In this section we will show that the wild part of the local monodromies at $0$ and $\infty$ of $K\ast L$ is determined by the local monodromies of $K$ and $L$. For $K\in\PPP$, we denote by $S(K)$ the closed subset of $\bar k^\star$ on which $K$ is not smooth, that is, the set of points $t\in\bar k^\star$ such that $K_{(t)}\neq 0$.

\begin{thm}\label{infinity}
 There exist bi-exact functors $\rho_{(0,\infty)}:\RR^w_0\times\RR^w_\infty\to\RR^w_\infty$, $\rho_{(\infty,\infty)}:\RR^w_\infty\times\RR^w_\infty\to\RR^w_\infty$ and $\rho_{(t,\infty)}:\RR_t\times\RR^w_\infty\to\RR^w_\infty$ for $t\in\bar k^\star$ such that, if we define $\rho_{(\infty,0)}(\FF,\GGG):=\rho_{(0,\infty)}(\GGG,\FF)$ and $\rho_{(\infty,t)}(\FF,\GGG):=\rho_{(t,\infty)}(\GGG,\FF)$, for every $K,L\in{\mathcal P}$ there is an isomorphism of $I_\infty$-representations
\begin{align*}
(K\ast L)_{(\infty)}^{w}&\cong\rho_{(\infty,\infty)}(K_{(\infty)}^w,L_{(\infty)}^w)\oplus\rho_{(0,\infty)}(K_{(0)}^w,L_{(\infty)}^w)\oplus\rho_{(\infty,0)}(K_{(\infty)}^w,L_{(0)}^w)\oplus
\\
&\oplus\left(\bigoplus_{s\in S(K)}\rho_{(s,\infty)}(K_{(s)},L_{(\infty)}^w)\right)\oplus\left(\bigoplus_{t\in S(L)}\rho_{(\infty,t)}(K_{(\infty)}^w,L_{(t)})\right).
\end{align*}
\end{thm}

Fix a non-trivial additive character $\psi:k\to\QQ^\star$. The key result that we will use to construct the local convolution functors is the following compatibility between convolution and Fourier transform with respect to $\psi$ \cite[Proposition 8.1.12]{katz1990esa}:
\begin{equation}\label{FT}
K\ast_! j^\star\LL_\psi[1]\cong j^\star\FT^\psi(j_!\istar K) 
\end{equation}
where $\iota:\GG_{m,\bar k}\to\GG_{m,\bar k}$ is the inversion map $t\mapsto t^{-1}$ and $j:\GG_{m,\bar k}\hookrightarrow \AAA^1_{\bar k}$ is the inclusion. If we denote by $\Phi$ the functor $K\mapsto K\ast j^\star\LL_\psi[1]$, it is clear by associativity that $\Phi^a(K)\ast\Phi^b(L)=\Phi^{a+b}(K\ast L)$ for every $a,b\geq 0$.

\begin{lem}\label{equivalence}
 The functor $\Phi:\PPP\to\PPP$ is an equivalence of categories, with quasi-inverse $\Psi:K\mapsto \istar j^\star\FT^{\bar\psi}(j_!K)$.
\end{lem}

\begin{proof}
 Let $K\in \PPP$, then 
$$
\Psi(\Phi(K))=\istar j^\star\FT^{\bar\psi}(j_!j^\star\FT^\psi(j_!\istar K)).
$$

If $k:\{0\}\to\AAA^1_{\bar k}$ is the inclusion, we have an exact triangle
$$
j_!j^\star\FT^\psi(j_!\istar K)\to\FT^\psi(j_!\istar K)\to k_\star k^\star\FT^\psi(j_!\istar K)\to
$$
which, after applying the triangulated functor $\FT^{\bar\psi}$ (which is the inverse of $\FT^\psi$), turns into
$$
\FT^{\bar\psi}(j_!j^\star\FT^\psi(j_!\istar K))\to j_!\istar K \to \FT^{\bar\psi}(k_\star k^\star\FT^\psi(j_!\istar K))\to.
$$

The last object is constant, being the Fourier transform with respect to $\bar\psi$ of a punctual object supported at $0$. So we get an isomorphism in $\PPP$
$$
\FT^{\bar\psi}(j_!j^\star\FT^\psi(j_!\istar K))\cong j_!\istar K
$$
and therefore
$$
\Psi(\Phi(K))\cong \istar j^\star j_! \istar K=\istar \istar K=K.
$$ 

In a similar way one can show that $\Phi(\Psi(K))\cong K$ in $\PPP$.
\end{proof}

We take $\rho_{(\infty,\infty)}$ to be the ``local convolution'' operator defined in \cite[Chapter 6]{katz1988gauss}: If $X:=\AAA^2_{(0,0)}$ (respectively $S:=\AAA^1_{(0)}$) denotes the strict henselization of $\AAA^2_{\bar k}$ at $(0,0)$ (resp. the strict henselization of $\AAA^1_{\bar k}$ at $0$) and $\bar\eta$ is a geometric generic point of $S$, then
$$
\rho_{(\infty,\infty)}(\FF,\GGG):=\HH^1(X\times_S\bar\eta,\FF\boxtimes\GGG)
$$
where the fibre product is taken with respect to the multiplication map $X\to S$, $(s,t)\mapsto st$. Then proposition \ref{infinity} holds for $K=\FF[1],L=\GGG[1]\in{\mathcal C}$ by \cite[6.6]{katz1988gauss}. Let $\mathcal T\subseteq\PPP$ be the subcategory of smooth objects which are tamely ramified at $0$ \cite[5.2]{katz1988gauss} (the truth of this property is independent of the choice of a representative for a class $K\in\PPP$). We will now show that proposition \ref{infinity} holds for objects in $\mathcal T$.

\begin{lem}\label{inT}
 Let $K=\FF[1]$ and $L=\GGG[1]$ be objects in ${\mathcal T}\subseteq{\mathcal P}$. Then 
$$
(K\ast L)_{(\infty)}^w\cong\rho_{(\infty,\infty)}(K_{(\infty)}^w,L_{(\infty)}^w).
$$
\end{lem}

\begin{proof}
 The proof is very similar to that of \cite[Theorem 6.5]{katz1988gauss}. For convenience we will review the main steps of the proof, indicating the differences where necessary. Consider the hypersurface $V\hookrightarrow \PP^2_{\bar k}\times\AAA^1_{\bar k}$ defined by the equation $XY=tZ^2$, and the projection $\pi:V\to\AAA^1_{\bar k}$. Then if $S$ is the henselization of $\AAA^1_{\bar k}$ at $0$, the map $V_S:=V\times_{\AAA^1_{\bar k}}S\to S$ gives a compactification of the multiplication map $\Gmk\times\Gmk\to\Gmk$ in a neighborhood of infinity, via the immersion $(s,t)\in\Gmk\times\Gmk\mapsto (s^{-1},t^{-1},1)\in\PP^2_{\bar k}$.

By the vanishing cycles exact triangle, the action of $I_\infty$ on $K\ast L$ is then given by the action of $I_0$ on $\R\Gamma(V_s,\R\Phi)$, where $V_s$ is the special fibre of the map $V_S\to S$ and $\R\Phi$ is the complex of vanishing cycles for the map $V_S\to S$ and the object $K_{Z/X}\otimes L_{Z/Y}=\FF_{Z/X}\otimes\GGG_{Z/Y}[2]$ extended by zero to $V_S$. 

The special fibre $V_s$ is the union of the lines $X=0$ and $Y=0$. By \cite[Lemma 6.5.3]{katz1988gauss} (which does not need $\FF$ or $\GGG$ to be totally wild at infinity) $I_0$ acts tamely on $\R\Phi$ along these lines, except perhaps at the points $(0,0,1)$, $(0,1,0)$ and $(1,0,0)$. Similarly, \cite[Lemma 6.5.4]{katz1988gauss} (which again does not use the hypothesis that $\FF$ and $\GGG$ are totally wild at infinity) shows that $I_0$ acts tamely on the stalks of $\R\Phi$ at the points $(1,0,0)$ and $(0,1,0)$.

Let $i:\{(0,0,1)\}\hookrightarrow V_s$ and $j:V_s\backslash\{(0,0,1)\}\hookrightarrow V_s$ the inclusions. We have an exact triangle
$$
j_!j^\star\R\Phi\to\R\Phi\to i_\star (\R\Phi_{(0,0,1)})\to
$$
of objects on $\Dbc(V_s,\QQ)$ with an $I_0$-action, which gives an exact triangle in the derived category of $\ell$-adic $I_0$-representations
\begin{equation}\label{eqvanishing}
\R\Gamma(V_s,j_!j^\star\R\Phi)\to\R\Gamma(V_s,\R\Phi)\to \R\Phi_{(0,0,1)}\to.
\end{equation}
We have seen that $I_0$ acts tamely on $j_!j^\star\R\Phi$, so it acts tamely on $\R\Gamma(V_s,j_!j^\star\R\Phi)$. It follows from the triangle above that there is an isomorphism between the wild part of the action of $I_0$ on $H^i(V_s,\R\Phi)$ (which is precisely $(K\ast L)^w_\infty$ for $i=-1$) and the wild part of the action of $I_0$ on $\R^i\Phi_{(0,0,1)}$ for every $i\in{\mathbb Z}$. By \cite[6.6]{katz1988gauss}, $\R^i\Phi_{(0,0,1)}$ depends only on $K_{(\infty)}$ and $L_{(\infty)}$. The functor that assigns to every pair $(\FF,\GGG)$ of $I_\infty$-representations their corresponding $\R^{1}\Phi(\FF,\GGG)_{(0,0,1)}$ preserves direct sums, and for $\FF$ and $\GGG$ totally wild it is Katz's ``local convolution''. In order to finish the proof, it remains to show that $\R^1\Phi(\FF^w,\GGG^w)$ is the wild part of $\R^1\Phi(\FF,\GGG)$, where $\FF^w$, $\GGG^w$ denote the wild parts of $\FF$ and $\GGG$ respectively.

Since $\R^1\Phi(\FF^w,\GGG^w)_{(0,0,1)}$ is totally wild by \cite[Theorem 6.5]{katz1988gauss} and $\FF=\FF^t\oplus\FF^w$, $\GGG=\GGG^t\oplus\GGG^w$ (where $\FF^t$ and $\GGG^t$ are the tame parts of $\FF$ and $\GGG$), it suffices to show that $\R^1\Phi(\FF,\GGG)_{(0,0,1)}$ is tame if either $\FF$ or $\GGG$ is tame. Since every tame $I_0$-representation is a succesive extension of tame characters $\LL_\chi$, it is enough to show that $\R^1\Phi(\LL_\chi,\GGG)_{(0,0,1)}$ is tame for any $\GGG$. By the triangle \eqref{eqvanishing}, this is equivalent to $\HH^1(V_s,\R\Phi(\LL_\chi,\GGG))$ being tame for any $\GGG$. But $\HH^1(V_s,\R\Phi(\LL_\chi,\GGG))$ is just $\HHH^{-1}(\LL_\chi[1]\ast_! L)$ as a representation of $I_\infty$, and it is known that $\LL_\chi[1]\ast_!\GGG$ is a succesive extension of Kummer objects \cite[Proposition 3.6.4]{gabber1996faisceaux}. In particular, it is a tame representation of $I_\infty$.
\end{proof}

Proposition \ref{slopes} gives the slopes of $\rho_{(\infty,\infty)}(\FF,\GGG)$ in terms of the slopes of $\FF$ and $\GGG$. By additivity, we may assume that $\FF$ and $\GGG$ have a single slope.

\begin{prop}\label{slopes1}
 If $\FF,\GGG\in\RR_\infty^w$ have single slopes $a$ and $b$ and ranks $m$ and $n$ respectively, then $\rho_{(\infty,\infty)}(\FF,\GGG)$ has a single slope $\frac{ab}{a+b}$ and rank $mn(a+b)$. 
\end{prop}

Next, we will define the functors $\rho_{(t,\infty)}$ for $t\in\bar k^\star$. Let $\iota:\PP^1_{\bar k}\to\PP^1_{\bar k}$ be the inversion map $t\mapsto t^{-1}$. It induces equivalences of categories, denoted $\istar$, between $\RR_t$ and $\RR_{t^{-1}}$ for every $t\in\PP^1_{\bar k}$. For $\FF\in\RR_t$ and $\GGG\in\RR_\infty^w$, we set
$$
\rho_{(t,\infty)}(\FF,\GGG)=\istar\FT_{(0,\infty)}^{-1}(\rho_{(\infty,\infty)}(\FT_{(t^{-1},\infty)}\istar \FF, \GGG))\in\RR_\infty^w
$$
where $\FT_{(0,\infty)}$ and $\FT_{(t^{-1},\infty)}$ are the local Fourier transform functors with respect to $\psi$ (cf. Theorem \ref{LFTT}).

The slopes and dimensions of $\rho_{(t,\infty)}(\FF,\GGG)$ can be given in terms of those of $\FF$ and $\GGG$. As before, we may assume that $\FF$ and $\GGG$ have single slopes.

\begin{prop}\label{slopes2}
 If $\FF\in\RR_t$, $\GGG\in\RR_\infty^w$ have single slopes $a\geq 0$ and $b>0$ and dimensions $m$ and $n$ respectively, then $\rho_{(t,\infty)}(\FF,\GGG)$ has a single slope $b$ and dimension $mn(a+1)$.  
\end{prop}

\begin{proof}
 This is immediate from the definition of $\rho_{(t,\infty)}$, proposition \ref{slopes1} and the properties of the local Fourier transform. By LFTT, $\FT_{(t^{-1},\infty)}\istar \FF\cong\LL_{\psi_{t^{-1}}}\otimes\FT_{(0,\infty)}\istar\FF$ has a single slope $1$ and rank $m(a+1)$. By proposition \ref{slopes1}, $\rho_{(\infty,\infty)}(\FT_{(t^{-1},\infty)}\istar \FF, \GGG)$ has slope $\frac{b}{b+1}$ and rank $mn(a+1)(b+1)$. Again by LFTT we conclude that $\rho_{(t,\infty)}(\FF,\GGG)$ has a single slope $b$ and rank $mn(a+1)(b+1)-mn(a+1)b=mn(a+1)$.
\end{proof}

We turn now to the definiton of $\rho_{(0,\infty)}$. By additivity it suffices to define $\rho_{(0,\infty)}(\FF,\GGG)$ for $\FF\in\RR^w_0$ with a single slope, since $\Hom_{\RR^w_0}(\FF,\FF')=\{0\}$ if $\FF$ and $\FF'$ have no slopes in common \cite[Proposition 1.1(4)]{katz1988gauss}. We start with the case where the slope of $\FF$ is $>1$. In that case, we define
$$
\rho_{(0,\infty)}(\FF,\GGG)=\istar\FT_{(0,\infty)}^{-1}(\rho_{(\infty,\infty)}(\FT_{(\infty,\infty)}\istar\FF,\GGG)).
$$

Now suppose that the slope of $\FF$ is $1$. By \cite[Lemma 8.5.7]{katz1988gauss} (applied to the irreducible components of $\FF$) there exist uniquely determined $t_1,\ldots,t_r\in \bar k^\star$ such that there is a canonical decomposition
$$
\FF\cong\bigoplus_{j=1}^r\FF_j\otimes \istar\LL_{\psi_{t_j}}
$$
with each $\FF_j$ having all slopes $<1$. We define
$$
\rho_{(0,\infty)}(\FF,\GGG)=\bigoplus_{j=1}^r \istar\FT_{(0,\infty)}^{-1}(\rho_{(t_j,\infty)}(\FT_{(\infty,0)}\istar\FF_j,\GGG))
$$

Finally, let $\FF$ have slope $a<1$. The functor $\FF\mapsto\FT_{(\infty,0)}\istar\FF$ is an equivalence of categories between $\RR_0^{w,<1}$ and $\RR_0^w$, and $\FT_{(\infty,0)}\istar\FF$ has slope $\frac{a}{1-a}=\frac{1}{a^{-1}-1}$. In particular, if $d=\lceil a^{-1}\rceil-1$, the $d$-th fold operation of the functor on $\FF$ has slope $\geq 1$. Therefore we can define recursively
$$
\rho_{(0,\infty)}(\FF,\GGG)=\istar\FT_{(0,\infty)}^{-1}(\rho_{(0,\infty)}(\FT_{(\infty,0)}\istar\FF,\GGG)).
$$

Once again we can determine the slopes and dimension of $\rho_{(0,\infty)}(\FF,\GGG)$ in terms of those of $\FF$ and $\GGG$. We may assume that $\FF$ and $\GGG$ have single slopes.

\begin{prop}\label{slopes3}
 If $\FF\in\RR^w_0$, $\GGG\in\RR_\infty^w$ have single slopes $a>0$ and $b>0$ and dimensions $m$ and $n$ respectively, then $\rho_{(0,\infty)}(\FF,\GGG)=0$ if $a\leq b$, and $\rho_{(0,\infty)}(\FF,\GGG)$ has a single slope $\frac{ab}{a-b}$ and dimension $mn(a-b)$ if $a>b$.  
\end{prop}

\begin{proof}
 Suppose that $a>1$. Then $\istar\FF$ has a single slope $a$, so $\FT_{(\infty,\infty)}\istar\FF$ has a single slope $\frac{a}{a-1}$ and dimension $m(a-1)$ by LFTT. By proposition \ref{slopes1}, $\rho_{(\infty,\infty)}(\FT_{(\infty,\infty)}\istar\FF,\GGG)$ has a single slope $\frac{ab}{ab+a-b}$ and dimension $mn(a-1)(\frac{a}{a-1}+b)=mn(ab+a-b)$. If $a\leq b$ then $\frac{ab}{ab+a-b}\geq 1$, so $\FT_{(0,\infty)}^{-1}(\rho_{(\infty,\infty)}(\FT_{(\infty,\infty)}\istar\FF,\GGG))=0$. Otherwise, $\FT_{(0,\infty)}^{-1}(\rho_{(\infty,\infty)}(\FT_{(\infty,\infty)}\istar\FF,\GGG))$ has slope $\frac{ab}{a-b}$ and dimension $mn(a-b)$ by LFTT.

Now suppose that $a=1$, and let $\FF\cong\bigoplus_{j=1}^r\FF_j\otimes \istar\LL_{\psi_{t_j}}$ with $\FF_j$ having slopes $<1$. By additivity, we may assume that $r=1$ and $\FF_1$ has a single slope $c<1$. Then $\FT_{(\infty,0)}\istar\FF_1$ has slope $\frac{c}{1-c}$ and dimension $m(1-c)$, so by proposition \ref{slopes2} $\rho_{(t_1,\infty)}(\FT_{(\infty,0)}\istar\FF_1,\GGG)$ has slope $b$ and dimension $mn(1-c)(1+\frac{c}{1-c})=mn$. We conclude that $\istar\FT_{(0,\infty)}^{-1}(\rho_{(t_j,\infty)}(\FT_{(\infty,0)}\istar\FF_1,\GGG))$ is $0$ if $b\geq 1$, and has slope $\frac{b}{1-b}$ and dimension $mn(1-b)$ if $b<1$ by LFTT.

It remains to prove it when $a<1$. Then $\FT_{(\infty,0)}\istar\FF$ has slope $\frac{a}{1-a}=\frac{1}{a^{-1}-1}$ and dimension $m(1-a)$. So we can assume, by induction on $\lceil a^{-1}\rceil$, that $\rho_{(0,\infty)}(\FT_{(\infty,0)}\istar\FF,\GGG)=0$ if $\frac{a}{1-a}\leq b$ and it has a single slope $\frac{ab}{a-b+ab}$ and dimension $mn(a-b+ab)$ otherwise. 

So if $\frac{a}{1-a}\leq b$ then
$$\rho_{(0,\infty)}(\FF,\GGG)=\istar\FT_{(0,\infty)}^{-1}(0)=0.$$
If $a\leq b<\frac{a}{1-a}$ then $\frac{ab}{a-b+ab}\geq 1$, so 
$$\rho_{(0,\infty)}(\FF,\GGG)=\istar\FT_{(0,\infty)}^{-1}(\rho_{(0,\infty)}(\FT_{(\infty,0)}\istar\FF,\GGG))=0.$$
Finally, if $a>b$ then $\rho_{(0,\infty)}(\FF,\GGG)=\istar\FT_{(0,\infty)}^{-1}(\rho_{(0,\infty)}(\FT_{(\infty,0)}\istar\FF,\GGG))$ has a single slope $\frac{ab}{a-b}$ and dimension $mn(a-b)$ by LFTT.
\end{proof}

The following cancellation lemmas will be useful later:

\begin{lem}\label{cancel}
 Let $\FF,\GGG\in\RR^w_\infty$. Then 
\begin{align*}&\istar\FT_{(0,\infty)}^{-1}(\rho_{(\infty,\infty)}(\FT_{(0,\infty)}\istar\FF,\GGG))\cong 
\\
&\cong \istar\FT_{(0,\infty)}^{-1}(\rho_{(\infty,\infty)}(\FF,\FT_{(0,\infty)}\istar\GGG))\cong\rho_{(\infty,\infty)}(\FF,\GGG).
\end{align*}
\end{lem}

\begin{proof}
 The statement is equivalent to $$\rho_{(\infty,\infty)}(\FT_{(0,\infty)}\istar\FF,\GGG)\cong\FT_{(0,\infty)}\istar\rho_{(\infty,\infty)}(\FF,\GGG).$$
By \cite[Theorem 1.5.6]{katz1986local} there exist smooth sheaves $\bar\FF,\bar\GGG$ on $\GG_{m,\bar k}$, tamely ramified at $0$, such that their monodromies at infinity are isomorphic to $\FF$ and $\GGG$ respectively. Then $\rho_{(\infty,\infty)}(\FT_{(0,\infty)}\istar\FF,\GGG)$ is the monodromy at infinity of $(j^\star\LL_\psi[1]\ast_!\bar \FF[1])\ast_!\bar\GGG[1]$ by \eqref{FT} and \cite[6.6]{katz1988gauss}, and $\FT_{(0,\infty)}\istar\rho_{(\infty,\infty)}(\FF,\GGG)$ is the monodromy at infinity of $j^\star\LL_\psi[1]\ast_!(\bar\FF[1]\ast_!\bar\GGG[1])$. We conclude by the associativity of the convolution.
\end{proof}

\begin{lem}\label{cancel2}
 Let $t\in\bar k^\star$, $\FF\in\RR_t$ and $\GGG\in\RR^w_\infty$. Then 
$$\istar\FT_{(0,\infty)}^{-1}(\rho_{(t,\infty)}(\FF,\FT_{(0,\infty)}\istar\GGG))\cong\rho_{(t,\infty)}(\FF,\GGG).$$ 
\end{lem}

\begin{proof}
 Immediate by lemma \ref{cancel} and the definition of $\rho_{(t,\infty)}$.
\end{proof}

\begin{lem}\label{cancel3}
 Let $\FF\in\RR^w_0$, $\GGG\in\RR^w_\infty$. Then
$$
\istar\FT^{-1}_{(0,\infty)}(\rho_{(0,\infty)}(\FF,\FT_{(0,\infty)}\istar\GGG))\cong\rho_{(0,\infty)}(\FF,\GGG).
$$
\end{lem}

\begin{proof}
 We can assume that $\FF$ has a single slope $a$. If $a\geq 1$ then it is a straightforward consequence of the definition of $\rho_{(0,\infty)}$ and lemmas \ref{cancel} and \ref{cancel2}. If $a<1$ we have, by induction on $\lceil a^{-1}\rceil-1$,
\begin{align*}
&\istar\FT^{-1}_{(0,\infty)}(\rho_{(0,\infty)}(\FF,\FT_{(0,\infty)}\istar\GGG))=
\\
&=\istar\FT^{-1}_{(0,\infty)}(\istar\FT^{-1}_{(0,\infty)}(\rho_{(0,\infty)}(\FT_{(\infty,0)}\istar\FF,\FT_{(0,\infty)}\istar\GGG)))\cong
\\
&\cong \istar\FT^{-1}_{(0,\infty)}(\rho_{(0,\infty)}(\FT_{(\infty,0)}\istar\FF,\GGG))=\rho_{(0,\infty)}(\FF,\GGG).\end{align*}
\end{proof}

Before proceeding with the proof of theorem \ref{infinity} we need an extra technical result

\begin{lem}
 For every $K\in\PPP$ there exists some integer $n\geq 0$ such that $\Phi^n(K)$ is in $\mathcal T$.
\end{lem}

\begin{proof}
 We proceed by induction on $d(K):=\lceil s^{-1}\rceil$, where $s$ is the infimum of the set of positive slopes of $K$ at $0$. If $d(K)=0$ then $K$ is tame at $0$, so $\Phi(K)\in{\mathcal T}$ by LFTT. If $d(K)=1$ then all slopes of $K$ at $0$ are $\geq 1$, so $\istar K$ has no $\infty$-slope in the interval $(0,1)$ and $\Phi(K)$ is tame at $0$ by LFTT. Therefore $\Phi(\Phi(K))=\Phi^2(K)\in{\mathcal T}$.

Otherwise, let $s=s_1 <\cdots < s_m<1$ be the slopes of $K$ at $0$ which are $<1$. Then the slopes of $\Phi(K)$ at $0$ are $\frac{s_1}{1-s_1}<\cdots <\frac{s_m}{1-s_m}$ by LFTT, so $d(\Phi(K))=\lceil\frac{1-s_1}{s_1}\rceil=\lceil s^{-1}-1\rceil=d(K)-1$. By induction hypothesis, there is some $n\geq 0$ such that $\Phi^n(\Phi(K))=\Phi^{n+1}(K)$ is in $\mathcal T$.
\end{proof}

\begin{proof}[Proof of theorem \ref{infinity}]
We proceed by induction on $m+n$, where $m$ (respectively $n$) is the smallest non-negative integer  such that $\Phi^m(K)$ (resp. $\Phi^n(L)$) is in $\mathcal T$. If $m+n=0$ then both $K$ and $L$ are in $\mathcal T$, so the result follows from lemma \ref{inT}. If $m+n>1$, we may assume by the commutativity of the convolution that $m>1$. Then $K\ast L\cong\Psi(\Phi(K\ast L))\cong\Psi(\Phi(K)\ast L)$. By induction hypothesis, we have 
\begin{align*}
(\Phi(K)\ast L)^w_{(\infty)}&\cong\rho_{(\infty,\infty)}(\Phi(K)_{(\infty)}^w,L_{(\infty)}^w)\oplus\rho_{(0,\infty)}(\Phi(K)_{(0)}^w,L_{(\infty)}^w)\oplus
\\
&\oplus\rho_{(\infty,0)}(\Phi(K)_{(\infty)}^w,L_{(0)}^w)\oplus\left(\bigoplus_{s\in S(\Phi(K))}\rho_{(s,\infty)}(\Phi(K)_{(s)},L_{(\infty)}^w)\right)\oplus
\\
&\oplus\left(\bigoplus_{t\in S(L)}\rho_{(\infty,t)}(\Phi(K)_{(\infty)}^w,L_{(t)})\right)
\end{align*}
so
\begin{align*}
(K\ast L)^w_{(\infty)}&=\istar\FT_{(0,\infty)}^{-1}(\Phi(K)\ast L)^w_{(\infty)}\cong
\\
&\cong \istar\FT_{(0,\infty)}^{-1}(\rho_{(\infty,\infty)}(\Phi(K)_{(\infty)}^w,L_{(\infty)}^w))\oplus \istar\FT_{(0,\infty)}^{-1}(\rho_{(0,\infty)}(\Phi(K)_{(0)}^w,L_{(\infty)}^w))\oplus
\\
&\oplus \istar\FT_{(0,\infty)}^{-1}(\rho_{(\infty,0)}(\Phi(K)_{(\infty)}^w,L_{(0)}^w))\oplus
\\
&\oplus\left(\bigoplus_{s\in S(\Phi(K))}\istar\FT_{(0,\infty)}^{-1}(\rho_{(s,\infty)}(\Phi(K)_{(s)},L_{(\infty)}^w))\right)\oplus
\\
&\oplus\left(\bigoplus_{t\in S(L)}\istar\FT_{(0,\infty)}^{-1}(\rho_{(\infty,t)}(\Phi(K)_{(\infty)}^w,L_{(t)}))\right).
\end{align*}

We will analyze these terms one by one. The first one is, by LFTT,
\begin{align*}
&\istar\FT_{(0,\infty)}^{-1}(\rho_{(\infty,\infty)}(\Phi(K)_{(\infty)}^w,L_{(\infty)}^w))=
\\
&=\istar\FT_{(0,\infty)}^{-1}(\rho_{(\infty,\infty)}(\FT_{(0,\infty)}\istar K_{(\infty)}^w,L_{(\infty)}^w))\oplus
\\
&\oplus \istar\FT_{(0,\infty)}^{-1}(\rho_{(\infty,\infty)}(\FT_{(\infty,\infty)}\istar K_{(0)}^{w,>1},L_{(\infty)}^w)) \oplus
\\
&\oplus\left(\bigoplus_{s\in S(K)}\istar\FT_{(0,\infty)}^{-1}(\rho_{(\infty,\infty)}(\FT_{(s^{-1},\infty)}\istar K_{(s)},L_{(\infty)}^w))\right)=
\\
&=\rho_{(\infty,\infty)}(K_{(\infty)}^w,L_{(\infty)}^w)\oplus\rho_{(0,\infty)}(K_{(0)}^{w,>1},L_{(\infty)}^w)\oplus\left(\bigoplus_{s\in S(K)}\rho_{(s,\infty)}(K_{(s)},L_{(\infty)}^w)\right)
\end{align*}
by lemma \ref{cancel} and the definitions of $\rho_{(0,\infty)}$ and $\rho_{(s,\infty)}$. For the second term, taking into account that $\Phi$ maps the slope $s<1$ part of the monodromy at $0$ of $K$ to the slope $\frac{s}{1-s}$ part of the monodromy at $0$ of $\Phi(K)$, we get
\begin{align*}
&\istar\FT_{(0,\infty)}^{-1}(\rho_{(0,\infty)}(\Phi(K)_{(0)}^{w},L_{(\infty)}^w))=
\\
&=\istar\FT_{(0,\infty)}^{-1}(\rho_{(0,\infty)}(\FT_{(\infty,0)}\istar K_{(0)}^{w,<1},L_{(\infty)}^w))=
\rho_{(0,\infty)}(K_{(0)}^{w,<1},L_{(\infty)}^w)
\end{align*}
by definition of $\rho_{(0,\infty)}$, where $K_{(0)}^{w,<1}$ denotes the part of positive slopes $<1$ of the monodromy of $K$ at $0$.

For the third term we split $L_{(0)}^w$ as a direct sum of its slopes $>1$ part $L_{(0)}^{w,>1}$, its slope $1$ part $L_{(0)}^{w,1}$ and its slopes $<1$ part $L_{(0)}^{w,<1}$. For the first one we have
\begin{align*}
&\istar\FT_{(0,\infty)}^{-1}(\rho_{(\infty,0)}(\Phi(K)_{(\infty)}^w,L_{(0)}^{w,>1}))=
\\
&=\istar\FT_{(0,\infty)}^{-1}(\istar\FT_{(0,\infty)}^{-1}(\rho_{(\infty,\infty)}(\Phi(K)_{(\infty)}^w,\FT_{(\infty,\infty)}\istar L_{(0)}^{w,>1}))).
\end{align*}
Now, using that $$\Phi(K)_{(\infty)}^w\cong\FT_{(0,\infty)}\istar K_{(\infty)}^w\oplus\FT_{(\infty,\infty)}\istar K_{(0)}^{w,>1}\oplus\bigoplus_{s\in S(K)}(\FT_{(s^{-1},\infty)}\istar K_{(s)})$$ we get
\begin{align*}
&\istar\FT_{(0,\infty)}^{-1}(\istar\FT_{(0,\infty)}^{-1}(\rho_{(\infty,\infty)}(\FT_{(0,\infty)}\istar K_{(\infty)}^w,\FT_{(\infty,\infty)}\istar L_{(0)}^{w,>1})))\oplus
\\
&\oplus \istar\FT_{(0,\infty)}^{-1}(\istar\FT_{(0,\infty)}^{-1}(\rho_{(\infty,\infty)}(\FT_{(\infty,\infty)}\istar K_{(0)}^{w,>1},\FT_{(\infty,\infty)}\istar L_{(0)}^{w,>1})))\oplus
\\
&\oplus \bigoplus_{s\in S(K)} \istar\FT_{(0,\infty)}^{-1}(\istar\FT_{(0,\infty)}^{-1}(\rho_{(\infty,\infty)}(\FT_{(s^{-1},\infty)}\istar K_{(s)},\FT_{(\infty,\infty)}\istar L_{(0)}^{w,>1}))).
\end{align*}
Notice that the second term in the direct sum vanishes: since both $\FT_{(\infty,\infty)}\istar K_{(0)}^{w,>1}$ and $\FT_{(\infty,\infty)}\istar L_{(0)}^{w,>1}$ have slopes $>1$, their $\rho_{(\infty,\infty)}$ has slopes $>\frac{1}{2}$ by proposition \ref{slopes1}. Then applying $\istar\FT_{(0,\infty)}^{-1}$ once gives slopes $>1$, so it vanishes when applying $\istar\FT_{(0,\infty)}^{-1}$ again. Similarly, the third term also vanishes (in this case $\FT_{(s^{-1},\infty)}\istar K_{(s)}$ has slope $1$). We conclude that
\begin{align*}
&\istar\FT_{(0,\infty)}^{-1}(\rho_{(\infty,0)}(\Phi(K)_{(\infty)}^w,L_{(0)}^{w,>1}))=
\\
&=\istar\FT_{(0,\infty)}^{-1}(\istar\FT_{(0,\infty)}^{-1}(\rho_{(\infty,\infty)}(\FT_{(0,\infty)}\istar K_{(\infty)}^w,\FT_{(\infty,\infty)}\istar L_{(0)}^{w,>1})))=
\\
&=\istar\FT_{(0,\infty)}^{-1}(\rho_{(\infty,\infty)}( K_{(\infty)}^w,\FT_{(\infty,\infty)}\istar L_{(0)}^{w,>1}))=\rho_{(\infty,0)}(K_{(\infty)}^w,L_{(0)}^{w,>1})
\end{align*}
by lemma \ref{cancel}. For the slope $1$ part of $L_{(0)}$ we get
\begin{align*}
&\istar\FT_{(0,\infty)}^{-1}(\rho_{(\infty,0)}(\Phi(K)_{(\infty)}^w,L_{(0)}^{w,1}))=
\\
&=\istar\FT_{(0,\infty)}^{-1}(\rho_{(\infty,0)}(\Phi(K)_{(\infty)}^{w,<1},L_{(0)}^{w,1}))=
\\
&=\bigoplus_{j=1}^r \istar\FT_{(0,\infty)}^{-1}(\istar\FT_{(0,\infty)}^{-1}(\rho_{(\infty,t_j)}(\FT_{(0,\infty)}\istar K_{(\infty)}^w,\FT_{(\infty,0)}\istar\FF_j)))
\end{align*}
if $L_{(0)}^{w,1}=\bigoplus_{j=1}^r\FF_j\otimes \istar\LL_{t_j}$ with all $\FF_j$ having slopes $<1$, since the functor $\rho_{(\infty,0)}(-,L_{(0)}^{w,1})$ vanishes for representations with slopes $\geq 1$ by proposition \ref{slopes3}. Now by lemma \ref{cancel2} we have
\begin{align*}
&\bigoplus_{j=1}^r \istar\FT_{(0,\infty)}^{-1}(\istar\FT_{(0,\infty)}^{-1}(\rho_{(\infty,t_j)}(\FT_{(0,\infty)}\istar K_{(\infty)}^w,\FT_{(\infty,0)}\istar\FF_j)))\cong
\\
&\cong\bigoplus_{j=1}^r \istar\FT_{(0,\infty)}^{-1}(\rho_{(\infty,t_j)}(K_{(\infty)}^w,\FT_{(\infty,0)}\istar\FF_j))=\rho_{(\infty,0)}(K_{(\infty)}^w,L_{(0)}^{w,1}).
\end{align*}
For the slope $<1$ part of $L_{(0)}^w$ we have, recursively,
\begin{align*}
&\istar\FT_{(0,\infty)}^{-1}(\rho_{(\infty,0)}(\Phi(K)_{(\infty)}^w,L_{(0)}^{w,<1}))=
\\
&=\istar\FT_{(0,\infty)}^{-1}(\istar\FT_{(0,\infty)}^{-1}(\rho_{(\infty,0)}(\Phi(K)_{(\infty)}^w,\FT_{(\infty,0)}\istar L_{(0)}^{w,<1})))=
\\
&=\istar\FT_{(0,\infty)}^{-1}(\rho_{(\infty,0)}(K^w_{(\infty)},\FT_{(\infty,0)}\istar L_{(0)}^{w,<1}))=\rho_{(\infty,0)}(K^w_{(\infty)},L_{(0)}^{w,<1}).
\end{align*}
We conclude that the third term is equal to
$$
\rho_{(\infty,0)}(K^w_{(\infty)},L_{(0)}^{w,>1})\oplus\rho_{(\infty,0)}(K^w_{(\infty)},L_{(0)}^{w,1})\oplus\rho_{(\infty,0)}(K^w_{(\infty)},L_{(0)}^{w,<1})=\rho_{(\infty,0)}(K^w_{(\infty)},L_{(0)}^{w}).
$$

We move now to the fourth term. The set $S(\Phi(K))\subseteq\bar k^\star$ of points where $\Phi(K)=\FT(\istar K)$ is not smooth is precisely the set of $s_j$ appearing in the decomposition $K_{(0)}^{w,1}\cong\bigoplus_{j=1}^r \FF_j\otimes \istar\LL_{\psi_{s_j}}$ with $\FF_j$ having slopes $<1$ \cite[Lemma 8.5.7]{katz1988gauss}. For every such $s_j$, 
$$
\istar\FT_{(0,\infty)}^{-1}(\rho_{(s_j,\infty)}(\Phi(K)_{(s_j)},L_{(\infty)}^w))=\istar\FT_{(0,\infty)}^{-1}(\rho_{(s_j,\infty)}(\FT_{(\infty,0)}\istar\FF_j,L_{(\infty)}^w)),
$$
so
\begin{align*}
&\bigoplus_{s\in S(\Phi(K))}\istar\FT_{(0,\infty)}^{-1}(\rho_{(s,\infty)}(\Phi(K)_{(s)},L_{(\infty)}^w))=
\\
&=\bigoplus_{j=1}^r \istar\FT_{(0,\infty)}^{-1}(\rho_{(s_j,\infty)}(\FT_{(\infty,0)}\istar\FF_j,L_{(\infty)}^w))=\rho_{(0,\infty)}(K_{(0)}^{w,1},L_{(\infty)}^w)
\end{align*}
by definition of $\rho_{(0,\infty)}$.

The last term is the direct sum, for $t\in S(L)$, of
\begin{align*}
&\istar\FT_{(0,\infty)}^{-1}(\rho_{(\infty,t)}(\Phi(K)_{(\infty)}^w,L_{(t)}))=
\\
&=\istar\FT_{(0,\infty)}^{-1}(\rho_{(\infty,t)}(\FT_{(0,\infty)}\istar K_{(\infty)}^w,L_{(t)}))\oplus
\\
&\oplus \istar\FT_{(0,\infty)}^{-1}(\rho_{(\infty,t)}(\FT_{(\infty,\infty)}\istar K_{(0)}^{w,>1},L_{(t)})) \oplus
\\
&\oplus\left(\bigoplus_{s\in S(K)}\istar\FT_{(0,\infty)}^{-1}(\rho_{(\infty,t)}(\FT_{(s^{-1},\infty)}\istar K_{(s)},L_{(t)}))\right).
\end{align*}
The last two summands vanish, since by proposition \ref{slopes2} $\rho_{(\infty,t)}(\FF,L_{(t)})$ has the same slopes as $\FF$, so its $\FT_{(0,\infty)}^{-1}$ is $0$ if these slopes are $\geq 1$. Therefore the last term comes down to
$$
\istar\FT_{(0,\infty)}^{-1}(\rho_{(\infty,t)}(\FT_{(0,\infty)}\istar K_{(\infty)}^w,L_{(t)}))=\rho_{(\infty,t)}(K_{(\infty)}^w,L_{(t)})
$$
by lemma \ref{cancel2}. The theorem is proved by taking the direct sum of all pieces and using the bi-exactness of the functors $\rho_{(-,-)}$.
\end{proof}

Using the fact that $\istar(K\ast L)\cong(\istar K)\ast(\istar L)$ for any $K,L\in\PPP$, Theorem \ref{infinity} immediately implies

\begin{thm}\label{zero}
 There exist bi-exact functors $\bar\rho_{(0,\infty)}:\RR^w_0\times\RR^w_\infty\to\RR^w_0$, $\bar\rho_{(0,0)}:\RR^w_0\times\RR^w_0\to\RR^w_0$ and $\bar\rho_{(t,0)}:\RR_t\times\RR^w_0\to\RR^w_0$ for $t\in\bar k^\star$ such that, if we define $\bar\rho_{(\infty,0)}(\FF,\GGG):=\bar\rho_{(0,\infty)}(\GGG,\FF)$ and $\bar\rho_{(0,t)}(\FF,\GGG)=\bar\rho_{(t,0)}(\GGG,\FF)$, for every $K,L\in{\mathcal P}$ there is an isomorphism of $I_0$-representations
\begin{align*}
(K\ast L)_{(0)}^{w}&\cong\bar\rho_{(0,0)}(K_{(0)}^w,L_{(0)}^w)\oplus\bar\rho_{(0,\infty)}(K_{(0)}^w,L_{(\infty)}^w)\oplus\bar\rho_{(\infty,0)}(K_{(\infty)}^w,L_{(0)}^w)\oplus
\\
&\oplus\left(\bigoplus_{s\in S(K)}\bar\rho_{(s,0)}(K_{(s)},L_{(0)}^w)\right)\oplus\left(\bigoplus_{t\in S(L)}\bar\rho_{(0,t)}(K_{(0)}^w,L_{(t)})\right).
\end{align*}
\end{thm}

The functors $\bar\rho_{(-,-)}$ have the same effect on the slopes and dimensions as their $\rho_{(-,-)}$ counterparts.

\section{Local monodromy at finite points of a convolution}

The finite points version of the main theorem is the following

\begin{thm}\label{finite}
 For every $u\in\bar k^\star$ there exist bi-exact functors $\rho_{(0,\infty)}^{(u)}:\RR^w_0\times\RR^w_\infty\to\RR_u$ and $\rho_{(s,t)}^{(u)}:\RR_s\times\RR_t\to\RR_u$ for $s,t\in\bar k^\star$ with $st=u$ such that, if we define $\rho_{(\infty,0)}^{(u)}(\FF,\GGG):=\rho_{(0,\infty)}^{(u)}(\GGG,\FF)$, for every $K,L\in{\mathcal P}$ there is an isomorphism of $I_u$-representations
$$
(K\ast L)_{(u)}\cong\rho_{(0,\infty)}(K_{(0)}^w,L_{(\infty)}^w)\oplus\rho_{(\infty,0)}(K_{(\infty)}^w,L_{(0)}^w)\oplus\bigoplus_{st=u}\rho_{(s,t)}^{(u)}(K_{(s)},L_{(t)}).
$$
\end{thm}

\begin{lem}
 Let $\FF,\GGG\in{\mathcal R}_\infty$ be two representations with all slopes $<1$, and let $a,b\in\bar k^\star$. Then $\istar\FT_{(0,\infty)}^{-1}(\rho_{(\infty,\infty)}(\LL_{\psi_a}\otimes\FF,\LL_{\psi_b}\otimes\GGG))\cong\LL_{\psi_{ab}}\otimes{\mathcal H}$ for some ${\mathcal H}\in{\mathcal R}_\infty$ with all slopes $<1$.
\end{lem}

\begin{proof}
 Since $\LL_{\psi_a}\otimes\FF$ and $\LL_{\psi_b}\otimes\GGG$ both have slope $1$ \cite[Lemma 1.3]{katz1988gauss}, their $\rho_{(\infty,\infty)}$ has slope $\frac{1}{2}$ by proposition \ref{slopes1}, so ${\mathcal A}:=\istar\FT_{(0,\infty)}^{-1}(\rho_{(\infty,\infty)}(\LL_{\psi_a}\otimes\FF,\LL_{\psi_b}\otimes\GGG))$ has slope $1$ by LFTT. By \cite[Lemma 8.5.7]{katz1988gauss} applied to its irreducible components, it suffices to show that $\LL_{\psi_c}\otimes{\mathcal A}$ has all slopes equal to $1$ for every $c\neq -ab$.

Let $m$ and $n$ be the dimensions of $\FF$ and $\GGG$. By \cite{katz1986local} there exist smooth $\QQ$-sheaves on $\GG_{m,\bar k}$ (also denoted by $\FF$ and $\GGG$), tamely ramified at $0$, which are isomorphic to $\FF$ and $\GGG$ as $I_\infty$-representations. Since $\LL_{\psi_a}\otimes\FF[1]$ and $\LL_{\psi_b}\otimes\GGG[1]$ are smooth on $\GG_{m,\bar k}$, tamely ramified at $0$ and totally wild at $\infty$ (with slope $1$), so is their convolution \cite[Theorem 5.1]{katz1988gauss} (with slope $\frac{1}{2}$). The rank of this convolution is $2mn$ by \cite[Theorem 5.1(4)]{katz1988gauss}. Then its (inverse) Fourier transform with respect to $\psi$ $\FT^{-1}((\LL_{\psi_a}\otimes\FF)[1]\ast(\LL_{\psi_b}\otimes\GGG)[1])$ has a unique positive slope $1$ at $0$ with multiplicity $mn$, is tame at infinity and smooth on $\GG_{m,\bar k}$ of rank $2mn$. For every $c\in\bar k^\star$ the $\infty$-slopes of $\LL_{\psi_c}\otimes{\mathcal A}$ are then $\leq 1$ (where ${\mathcal A}=\istar\FT^{-1}((\LL_{\psi_a}\otimes\FF)[1]\ast(\LL_{\psi_b}\otimes\GGG)[1])$) \cite[Lemma 1.3]{katz1988gauss}, and they are all $=1$ if and only if its Swan conductor at $\infty$ (which is its Euler characteristic by the Ogg-Shafarevic formula) is $2mn$. We have, by corollary \ref{lema1cor},
\begin{align*}
\chi(\GG_{m,\bar k},\LL_{\psi_c}\otimes{\mathcal A})&=\chi(\GG_{m,\bar k},\LL_{\psi_c}\otimes(\istar\LL_{\bar\psi}[1]\ast (\LL_{\psi_a}\otimes\FF)[1]\ast(\LL_{\psi_b}\otimes\GGG)[1]))=
\\
&=\chi(\GG_{m,\bar k},\tau_c^\star\LL_{\psi}\otimes(\istar\LL_{\bar\psi}[1]\ast (\LL_{\psi_a}\otimes\FF)[1]\ast(\LL_{\psi_b}\otimes\GGG)[1]))=
\\
&=\chi(\GG_{m,\bar k},\LL_{\psi}\otimes\tau_{1/c}^\star(\istar\LL_{\bar\psi}[1]\ast (\LL_{\psi_a}\otimes\FF)[1]\ast(\LL_{\psi_b}\otimes\GGG)[1]))=
\\
&=\chi(\GG_{m,\bar k},\LL_{\psi}\otimes(\istar\LL_{\psi_{-1}}[1]\ast (\LL_{\psi_{a}}\otimes\FF)[1]\ast(\LL_{\psi_{b/c}}\otimes\tau_{1/c}^\star\GGG)[1]))=
\\
&=\chi(\GG_{m,\bar k},(\istar\LL_{\psi_{-1}}[1]\ast (\LL_{\psi_{a}}\otimes\FF)[1])\otimes\FT(\LL_{\psi_{b/c}}\otimes\tau_{1/c}^\star\GGG)[1])=
\\
&=\chi(\GG_{m,\bar k},\istar\tau_{-1}^\star(\LL_{\psi}[1]\ast \istar(\LL_{\psi_{a}}\otimes\FF)[1])\otimes\FT(\LL_{\psi_{b/c}}\otimes\tau_{1/c}^\star\GGG)[1])=
\\
&=\chi(\GG_{m,\bar k},\istar\tau_{-1}^\star\FT(\LL_{\psi_{a}}\otimes\FF)[1]\otimes\FT(\LL_{\psi_{b/c}}\otimes\tau_{1/c}^\star\GGG)[1])
\end{align*}
where $\tau_\lambda:\Gmk\to\Gmk$ is the multiplication by $\lambda$ map, and we have made repeated use of the fact that $(\tau_a^\star K)\ast(\tau_b^\star L)\cong\tau_{ab}^\star(K\ast L)$ \cite[8.1.10]{katz1990esa}.

The objects $\FT(\LL_{\psi_{a}}\otimes\FF)$ and $\FT(\LL_{\psi_{b/c}}\otimes\tau_{1/c}^\star\GGG)$ are smooth on $\GG_{m,\bar k}$ (except at $-a$ and $-\frac{b}{c}$ respectively) of ranks $m$ and $n$ and tame at $0$ and infinity by LFTT. Suppose that $c\neq -ab$. Then the Euler characteristic of $\istar\tau_{-1}^\star\FT(\LL_{\psi_{a}}\otimes\FF)\otimes\FT(\LL_{\psi_{b/c}}\otimes\tau_{1/c}^\star\GGG)$ on $\Gmk$ is the sum of two local terms. At $\frac{1}{a}$ the second factor is smooth of rank $n$, so the local term at this point is $n$ times the corresponding local term for the Euler characteristic of $\istar\tau_{-1}^\star\FT(\LL_{\psi_{a}}\otimes\FF)$ (the sum of the drop of the rank and the Swan conductor). Similarly, at $-\frac{b}{c}$ the first term is smooth of rank $m$, so the local term is $m$ times the corresponding local term for the Euler characteristic of $\FT(\LL_{\psi_{b/c}}\otimes\tau_{1/c}^\star\GGG)$. We conclude that the Euler characteristic of the tensor product is
\begin{align*}
&n\cdot\chi(\Gmk,\istar\tau_{-1}^\star\FT(\LL_{\psi_{a}}\otimes\FF))+m\cdot\chi(\Gmk,\FT(\LL_{\psi_{b/c}}\otimes\tau_{1/c}^\star\GGG))=
\\
&=n\cdot\chi(\Gmk,\FT(\LL_{\psi_{a}}\otimes\FF))+m\cdot\chi(\Gmk,\FT(\LL_{\psi_{b/c}}\otimes\tau_{1/c}^\star\GGG))=
\\
&=n\cdot\chi(\Gmk,\LL_{\psi_{a}}\otimes\FF)+m\cdot\chi(\Gmk,\LL_{\psi_{b/c}}\otimes\tau_{1/c}^\star\GGG)=-nm-mn=-2mn
\end{align*}
since, for $K\in\Dbc(\Gmk,\QQ)$, we have 
\begin{align*}
 \chi(\Gmk,\FT(K))&=\chi(\AAA^1_{\bar k},\FT(j_!K))-\mathrm{rank}_0(\FT(j_!K))=
\\ &=-\mathrm{rank}_0(j_!K)+\chi(\AAA^1_{\bar k},j_!K)=\chi(\Gmk,K).
\end{align*}
\end{proof}

We can now define the functors $\rho_{(s,t)}^{(u)}:\RR_s\times\RR_t\to\RR_u$ for every $s,t,u\in\bar k^\star$. We set
$$
\rho_{(s,t)}^{(u)}(\FF,\GGG)=\istar\FT_{(u^{-1},\infty)}^{-1}( \istar\FT_{(0,\infty)}^{-1}(\rho_{(\infty,\infty)}(\FT_{(s^{-1},\infty)}\istar\FF,\FT_{(t^{-1},\infty)}\istar\GGG))).
$$
The previous lemma implies that $\rho_{(s,t)}^{(u)}=0$ if $u\neq st$.

Next, we define the functor $\rho_{(0,\infty)}^{(u)}:\RR_0^w\times\RR_\infty^w\to\RR_u$ for every $u\in\bar k^\star$. Since $\RR_0^w=\bigoplus_{\lambda>0}\RR_0^{w,\lambda}$ and similarly for $\RR_\infty^w$, it suffices to define $\rho_{(0,\infty)}^{(u)}(\FF,\GGG)$ for $\FF,\GGG$ having single slopes $a,b$. If $a>1$, we set
$$
\rho_{(0,\infty)}^{(u)}(\FF,\GGG)= \istar\FT_{(u^{-1},\infty)}^{-1}(\rho_{(\infty,\infty)}(\FT_{(\infty,\infty)}\istar\FF,\GGG))
$$
If $a=1$,
$$
\rho_{(0,\infty)}^{(u)}(\FF,\GGG)= \istar\FT_{(u^{-1},\infty)}^{-1}(\bigoplus_{t\in\bar k^\star}\rho_{(t,\infty)}(\FT_{(\infty,t)}\istar\FF,\GGG))
$$
 where the last sum is finite, since there are only finitely many $t\in\bar k^\star$ such that $\FT_{(\infty,t)}\istar\FF\neq 0$. If $a<1$ we define
$$
\rho_{(0,\infty)}^{(u)}(\FF,\GGG)=\istar\FT_{(u^{-1},\infty)}^{-1}(\rho_{(0,\infty)}(\FT_{(\infty,0)}\istar\FF,\GGG)).
$$
Notice that in all cases $\rho_{(0,\infty)}^{(u)}(\FF,\GGG)=0$ if $a\neq b$, since $\FT^{-1}_{(u^{-1},\infty)}$ can only be non-zero for representations with slope $1$.

\begin{proof}[Proof of theorem \ref{finite}]
 We have $K\ast L\cong\Psi(\Phi(K\ast L))\cong\Psi(\Phi(K)\ast L)$. By theorem \ref{infinity}
\begin{align*}
(\Phi(K)\ast L)^w_{(\infty)} &\cong  \rho_{(\infty,\infty)}(\Phi(K)_{(\infty)}^w,L_{(\infty)}^w)\oplus
\\ 
&\oplus\rho_{(0,\infty)}(\Phi(K)_{(0)}^w,L_{(\infty)}^w)\oplus\rho_{(\infty,0)}(\Phi(K)_{(\infty)}^w,L_{(0)}^w)\oplus \\  &\oplus\left(\bigoplus_{s\in S(\Phi(K))}\rho_{(s,\infty)}(\Phi(K)_{(s)},L_{(\infty)}^w)\right)\oplus\left(\bigoplus_{t\in S(L)}\rho_{(\infty,t)}(\Phi(K)_{(\infty)}^w,L_{(t)})\right).
\end{align*}
We are only interested in the slope $1$ part, which gives rise to non-trivial monodromy at finite points after applying $\Psi$. 

If $\FF$ has slope $a\leq 1$ and $\GGG$ has slope $b$, then by proposition \ref{slopes1} $\rho_{(\infty,\infty)}(\FF,\GGG)$ has slope $(a^{-1}+b^{-1})^{-1}\leq(1+b^{-1})^{-1}=\frac{b}{b+1}<1$. Similarly, if $a\geq 1$ $\rho_{(\infty,0)}(\FF,\GGG)$ (if non-zero) has slope $(a^{-1}-b^{-1})^{-1}\geq(1-b^{-1})^{-1}=\frac{b}{b-1}>1$. On the other hand, by proposition \ref{slopes2} $\rho_{(\infty,t)}(\FF,L_{(t)})$ has the same slopes as $\FF$. We conclude that the slope $1$ part of $(\Phi(K)\ast L)^{w,1}_{(\infty)}$ is the slope $1$ part of
\begin{align*}
&\rho_{(\infty,\infty)}(\Phi(K)_{(\infty)}^{w,>1},L_{(\infty)}^w)\oplus\rho_{(0,\infty)}(\Phi(K)_{(0)}^w,L_{(\infty)}^w)\oplus\rho_{(\infty,0)}(\Phi(K)_{(\infty)}^{w,<1},L_{(0)}^w)\oplus
\\
&\oplus\left(\bigoplus_{s\in S(\Phi(K))}\rho_{(s,\infty)}(\Phi(K)_{(s)},L_{(\infty)}^{w,1})\right)\oplus\left(\bigoplus_{t\in S(L)}\rho_{(\infty,t)}(\Phi(K)_{(\infty)}^{w,1},L_{(t)})\right)=
\\
&=\rho_{(\infty,\infty)}(\FT_{(\infty,\infty)}(\istar K_{(0)}^{w,>1}),L_{(\infty)}^w)\oplus\rho_{(0,\infty)}(\FT_{(\infty,0)}(\istar K_{(0)}^{w,<1}),L_{(\infty)}^w)\oplus
\\
&\oplus\rho_{(\infty,0)}(\FT_{(0,\infty)}(\istar K_{(\infty)}^{w}),L_{(0)}^w)\oplus\left(\bigoplus_{s\in \bar k^\star}\rho_{(s,\infty)}(\FT_{(\infty,s)}(\istar K_{(0)}^{w,1}),L_{(\infty)}^{w,1})\right)\oplus
\\
&\oplus\left(\bigoplus_{s,t\in \bar k^\star}\rho_{(\infty,t)}(\FT_{(s^{-1},\infty)}(\istar K_{(s)}),L_{(t)})\right)
\end{align*}
so, by definition of $\rho_{(0,\infty)}^{(u)}$, $\rho_{(s,t)}^{(u)}$ and $\rho_{(\infty,t)}$,
\begin{align*} 
(K\ast L)_{(u)}&=\Psi(\Phi(K)\ast L)_{(u)}=\istar\FT_{(u^{-1},\infty)}^{-1}(\Phi(K)\ast L)^{w,1}_{(\infty)}=
\\
&=\istar\FT_{(u^{-1},\infty)}^{-1}(\rho_{(\infty,\infty)}(\FT_{(\infty,\infty)}(\istar K_{(0)}^{w,>1}),L_{(\infty)}^w))\oplus 
\\
&\oplus \istar\FT_{(u^{-1},\infty)}^{-1}(\rho_{(0,\infty)}(\FT_{(\infty,0)}(\istar K_{(0)}^{w,<1}),L_{(\infty)}^w))\oplus
\\
&\oplus \istar\FT_{(u^{-1},\infty)}^{-1}(\rho_{(\infty,0)}(\FT_{(0,\infty)}(\istar K_{(\infty)}^w),L_{(0)}^w))\oplus
\\
&\oplus \left(\bigoplus_{s\in \bar k^\star}\istar\FT_{(u^{-1},\infty)}^{-1}(\rho_{(s,\infty)}(\FT_{(\infty,s)}(\istar K_{(0)}^{w,1}),L_{(\infty)}^{w,1}))\right)\oplus
\\
&\oplus\left(\bigoplus_{s,t\in \bar k^\star}\istar\FT_{(u^{-1},\infty)}^{-1}(\rho_{(\infty,t)}(\FT_{(s^{-1},\infty)}(\istar K_{(s)}),L_{(t)}))\right)=
\\
&=\rho^{(u)}_{(0,\infty)}(K_{(0)}^{w,>1},L_{(\infty)}^w)\oplus \rho^{(u)}_{(0,\infty)}(K_{(0)}^{w,<1},L_{(\infty)}^w)\oplus
\\
&\oplus \istar\FT_{(u^{-1},\infty)}^{-1}(\rho_{(\infty,0)}(\FT_{(0,\infty)}(\istar K_{(\infty)}^w),L_{(0)}^w))\oplus\rho^{(u)}_{(0,\infty)}(K_{(0)}^{w,1},L_{(\infty)}^{w,1})\oplus
\\
&\oplus\left(\bigoplus_{s,t\in \bar k^\star}\istar\FT_{(u^{-1},\infty)}^{-1}(\istar\FT_{(0,\infty)}^{-1}(\rho_{(\infty,\infty)}(\FT_{(s^{-1},\infty)}(\istar K_{(s)}),\FT_{(t^{-1},\infty)}(\istar L_{(t)}))))\right)=
\\
&=\rho^{(u)}_{(0,\infty)}(K_{(0)}^{w},L_{(\infty)}^w)\oplus\left(\bigoplus_{s,t\in \bar k^\star}\rho_{(s,t)}^{(u)}(K_{(s)},L_{(t)})\right)\oplus
\\
&\oplus \istar\FT_{(u^{-1},\infty)}^{-1}(\rho_{(\infty,0)}(\FT_{(0,\infty)}(\istar K_{(\infty)}^w),L_{(0)}^w)). &
\end{align*}

It only remains to show that the last term is equal to $\rho^{(u)}_{(\infty,0)}(K_{(\infty)}^{w},L_{(0)}^w)$. We have, using lemmas \ref{cancel}, \ref{cancel2} and \ref{cancel3},
\begin{align*}
&\istar\FT_{(u^{-1},\infty)}^{-1}(\rho_{(\infty,0)}(\FT_{(0,\infty)}(\istar K_{(\infty)}^w),L_{(0)}^w))=
\\
&=\istar\FT_{(u^{-1},\infty)}^{-1}(\rho_{(\infty,0)}(\FT_{(0,\infty)}(\istar K_{(\infty)}^w),L_{(0)}^{w,>1}))\oplus
\\
&\oplus \istar\FT_{(u^{-1},\infty)}^{-1}(\rho_{(\infty,0)}(\FT_{(0,\infty)}(\istar K_{(\infty)}^w),L_{(0)}^{w,1}))\oplus
\\
&\oplus \istar\FT_{(u^{-1},\infty)}^{-1}(\rho_{(\infty,0)}(\FT_{(0,\infty)}(\istar K_{(\infty)}^w),L_{(0)}^{w,<1}))=
\\
&=\istar\FT_{(u^{-1},\infty)}^{-1}(\istar\FT^{-1}_{(0,\infty)}(\rho_{(\infty,\infty)}(\FT_{(0,\infty)}(\istar K_{(\infty)}^w),\FT_{(\infty,\infty)}\istar L_{(0)}^{w,>1})))\oplus
\\
&\oplus \istar\FT_{(u^{-1},\infty)}^{-1}(\bigoplus_{t\in\bar k^\star}\istar\FT_{(0,\infty)}^{-1}(\rho_{(\infty,t)}(\FT_{(0,\infty)}(\istar K_{(\infty)}^w),\FT_{(\infty,t)}\istar L_{(0)}^{w,1})))\oplus
\\
&\oplus \istar\FT_{(u^{-1},\infty)}^{-1}(\istar\FT_{(0,\infty)}^{-1}(\rho_{(\infty,0)}(\FT_{(0,\infty)}(\istar K_{(\infty)}^w),\FT_{(\infty,0)}\istar L_{(0)}^{w,<1})))=
\\
&=\istar\FT_{(u^{-1},\infty)}^{-1}(\rho_{(\infty,\infty)}(K_{(\infty)}^w,\FT_{(\infty,\infty)}\istar L_{(0)}^{w,>1}))\oplus
\\
&\oplus \istar\FT_{(u^{-1},\infty)}^{-1}(\bigoplus_{t\in\bar k^\star}\rho_{(\infty,t)}(K_{(\infty)}^w,\FT_{(\infty,t)}\istar L_{(0)}^{w,1}))\oplus
\\
&\oplus \istar\FT_{(u^{-1},\infty)}^{-1}(\rho_{(\infty,0)}(\istar K_{(\infty)}^w,\FT_{(\infty,0)}\istar L_{(0)}^{w,<1}))=
\\
&=\rho_{(\infty,0)}^{(u)}(K_{(\infty)}^w,L_{(0)}^{w,>1})\oplus\rho_{(\infty,0)}^{(u)}(K_{(\infty)}^w,L_{(0)}^{w,1})\oplus\rho_{(\infty,0)}^{(u)}(K_{(\infty)}^w,L_{(0)}^{w,<1})=
\\
&=\rho_{(\infty,0)}^{(u)}(K_{(\infty)}^w,L_{(0)}^w).
\end{align*}
\end{proof}

The following result generalizes \cite[Lemma 19.5]{katz2010mellin}.

\begin{cor}
 Let $K,L\in\PPP$. Suppose that \begin{enumerate} 
\item $K_{(0)}$ and $L_{(\infty)}$ do not have any positive slope in common and
\item $K_{(\infty)}$ and $L_{(0)}$ do not have any positive slope in common. 
\end{enumerate}
 Then $S(K\ast L)=S(K)\cdot S(L)$. 
\end{cor}

\begin{proof}
 In this case both $\rho_{(0,\infty)}^{(u)}(K_{(0)},L_{(\infty)})$ and $\rho_{(\infty,0)}^{(u)}(K_{(\infty)},L_{(0)})$ vanish for every $u\in \bar k^\star$, so 
$$
(K\ast L)_{(u)}=\bigoplus_{st=u}\rho^{(u)}_{(s,t)}(K_{(s)},L_{(t)})
$$
which vanishes if and only if $u\notin S(K)\cdot S(L)$.
\end{proof}

For every $s\in\bar k^\star$ and every $\FF\in\RR_s$ there exists a semisimple $K\in\PPP$ without punctual part, smooth on $\GG_{m,\bar k}-\{s\}$, tamely ramified at $0$ and $\infty$ and such that $K_{(s)}\cong\FF$. It suffices to prove it when $\FF$ is either a single tame Jordan block or totally wild irreducible, since every representation is a direct sum of those. In the first case, if $\FF$ has character $\chi\neq{\mathbf 1}$ and dimension $n$ we can take $j_\star \phi^\star{\mathcal H}(\psi;n\;{\mathbf 1}'s;n\;\chi's)[1]$, where ${\mathcal H}(\psi;n\;{\mathbf 1}'s;n\;\chi's)$ is the hypergeometric sheaf associated to the $n$-uples $(\mathbf 1,\ldots,\mathbf 1)$ and $(\chi,\ldots,\chi)$, $j:\GG_{m,\bar k}-\{s\}\hookrightarrow\GG_{m,\bar k}$ is the inclusion and $\phi(x)=\frac{1}{s-x}$ \cite[8.4]{katz1990esa}. If $\FF$ has trivial character (i.e. it is unipotent) we take $j_\star \phi^\star{\mathcal H}(\psi;n+1\;\chi's;n+1\;{\mathbf 1}'s)[1]$ where $\chi$ is any non-trivial character \cite[Theorem 8.4.2]{katz1990esa}. In the second case, we take $j_\star \phi^\star{\mathcal M}[1]$, where $\mathcal M$ is a smooth sheaf on $\GG_{m,\bar k}$, tame at $0$, such that its monodromy at infinity is isomorphic to $\FF$ viewed as a representation of $I_\infty$ via the isomorphism $I_s\cong I_\infty$ induced by $\phi$ (such an $\mathcal M$ exists by \cite{katz1986local}).

\begin{cor}\label{translate} If we identify $I_u$ and $I_1$ via the isomorphism mapping the uniformizer $x-u$ of $I_u$ to the uniformizer $x-1$ of $I_1$, then $\tau_{st}^\star\rho^{(st)}_{(s,t)}(\FF,\GGG)\cong\rho^{(1)}_{(1,1)}(\tau_s^\star\FF,\tau_t^\star\GGG)$ for every $\FF\in\RR_s$, $\GGG\in\RR_t$, where $\tau_\lambda:I_1\to I_1$ is the automorphism induced by $(x-1)\mapsto \lambda(x-1)$. In particular, if $\FF$ and $\GGG$ are tame then $\rho^{(st)}_{(s,t)}(\FF,\GGG)\cong\rho^{(1)}_{(1,1)}(\FF,\GGG)$
\end{cor}

\begin{proof}
 Let $K,L\in\PPP$ be smooth on $\GG_{m,\bar k}-\{s\}$ and $\GG_{m,\bar k}-\{t\}$ respectively, tame at $0$ and $\infty$ and such that $K_{(s)}\cong\FF$, $K_{(t)}\cong\GGG$. Then by theorem \ref{finite} $(K\ast L)_{(st)}\cong\rho_{(s,t)}^{(st)}(\FF,\GGG)$.

Let $\tau_\lambda:\GG_{m,\bar k}\to\GG_{m,\bar k}$ be the multiplication by $\lambda$ map. Under the given identification $I_\lambda\cong I_1$, it induces $\tau_\lambda$ on $I_1$. The objects $\tau_s^\star K$ and $\tau_t^\star L$ are smooth on $\GG_{m,\bar k}-\{1\}$, tame at $0$ and $\infty$ and $(\tau_s^\star K)_{(1)}=\tau_s^\star(K_{(s)})$, $(\tau_t^\star L)_{(1)}=\tau_t^\star(L_{(t)})$. Using that $\tau_\lambda(K\ast L)\cong(\tau_\lambda K)\ast L$ \cite[8.1.10]{katz1990esa}, we conclude that
\begin{align*}
\rho^{(1)}_{(1,1)}(\tau_s^\star\FF,\tau_t^\star\GGG)&\cong(\tau_s^\star K\ast\tau_t^\star L)_{(1)}=(\tau_{st}^\star(K\ast L))_{(1)}=
\\
&=\tau_{st}^\star((K\ast L)_{(st)})=\tau_{st}^\star\rho_{(s,t)}^{(st)}(K_{(s)},L_{(t)})=\tau_{st}^\star\rho_{(s,t)}^{(st)}(\FF,\GGG).
\end{align*}

The last statement follows from the fact that $\tau_\lambda^\star\FF\cong\FF$ for every $\lambda\in\bar k^\star$ is $\FF$ is tame (since every such $\FF$ is a succesive extension of tame characters of the form $\LL_\chi$).
  
\end{proof}

This reduces the study of the properties of the functors $\rho_{(s,t)}^{(st)}$ to those of $\rho_{(1,1)}^{(1)}$. We now give an alternative local description of this functor, which is more convenient for some applications.

\begin{lem}\label{vanishing}
 Let $S$ be the henselization of $\GG_{m,\bar k}$ at $1$, and $\mu_S:(\GG_{m,\bar k}\times\GG_{m,\bar k})\times_{\GG_{m,\bar k}} S \to S$ the map induced by multiplication. Then for every semisimple $K,L\in\PPP$ which are smooth on $\GG_{m,\bar k}-\{1\}$ and tame at $0$ and $\infty$ we have
$$
\rho_{(1,1)}^{(1)}(K_{(1)},L_{(1)})\cong\R^{-1}\Phi_{(1,1)}
$$
where $\R\Phi$ is the vanishing cycles complex for the object $K\boxtimes L$ relative to $\mu_S$.
\end{lem}

\begin{proof}
 By additivity we may assume that $K$ and $L$ are irreducible. If one of them is punctual the result is trivial, since $\delta_1$ (the punctual perverse sheaf at $1$) is the identity for the convolution. So we will assume that $K$ and $L$ are irreducible middle extensions.

 Let $V\subseteq\PP^2_{\bar k}\times\GG_{m,\bar k}$ (with coordinates $((X,Y,Z),t)$) be the subscheme defined by $XY=tZ^2$. The projection $\pi:V\to\GG_{m,\bar k}$ is a compactification of the multiplication map $\GG_{m,\bar k}\times\GG_{m,\bar k}\to\GG_{m,\bar k}$. On $V$ we consider the object $M$, extension by zero of $K_{X/Z}\otimes L_{Y/Z}$ on the open set $XYZ\neq 0$. Then by theorem \ref{finite} $\rho_{(1,1)}^{(1)}(K_{(1)},L_{(1)})$ is the $(-1)$-st cohomology group of the mapping cone of the specialization map $(\R\pi_\star M)_1\to(\R\pi_\star M)_{\bar\eta}$.

On the fibre over any $t\in \bar k^\star$, the object $M$ is smooth except at the four points $(0,1,0),(1,0,0),(1,t,1)$ and $(t,1,1)$, which are all distinct except for the last two when $t=1$. At $(0,1,0)$ and $(1,0,0)$ it is tamely ramified and the stalk vanishes, since $K$ and $L$ are tamely ramified at $0$ and $\infty$. If $t\neq 1$ $L_{Y/Z}$ is smooth at $(1,t,1)$, so the drop of the rank plus the Swan conductor of $\HHH^{-2}(M)$ at $(1,t,1)$ is that of $\HHH^{-1}(K_{X/Z})$ at $(1,t,1)$ multiplied by $n$, that is, $n$ times the sum of the drop of the rank and the Swan conductor of $\HHH^{-1}(K)$ at $1$. In particular it is independent of $t\neq 1$. Similarly, the drop of the rank plus the Swan conductor of $\HHH^{-2}(M)$ at $(t,1,1)$ is independent of $t\neq 1$. By \cite[Th\'eor\`eme 2.1.1]{laumonsemi} we conclude that $\HHH^{-2}(M)$ (and therefore $M\cong\HHH^{-2}(M)[2]$) is universally locally acyclic on $V$ for $\pi$, except perhaps at the point $((1,1,1),1)$. In particular, the vanishing cycles complex $\R\tilde\Phi$ for the map $\pi:V\times_{\GG_{m,\bar k}}S\to S$ is punctual supported on $(1,1,1)$, so $\rho_{(1,1)}^{(1)}(K_1,L_1)\cong\R^{-1}\tilde\Phi_{(1,1,1)}=\R^{-1}\Phi_{(1,1)}$ by the vanishing cycles exact sequence \cite[Expos\'e XIII, 2.1.8.9]{deligne1972groupes}
$$
0\to(\R^{-1}\pi_\star M)_1\to(\R^{-1}\pi_\star M)_{\bar\eta}\to\HH^{-1}(V_1,\R\tilde\Phi)=\R^{-1}\tilde\Phi_{(1,1,1)}\to(\R^{0}\pi_\star M)_1\to 0.
$$
\end{proof}

\begin{cor}\label{oneofthemistame}
 For every $\FF\in\RR_1^t$ and $\GGG\in\RR_1$, if we view them as elements of $\RR_0$ via the translation $t\mapsto t+1$, we have
$$
\rho_{(1,1)}^{(1)}(\FF,\GGG)\cong\FT_{(0,\infty)}^{-1}((\FT_{(0,\infty)}\FF)\otimes(\FT_{(0,\infty)}\GGG)).
$$
In particular, if $\FF$ has dimension $m$ and $\GGG$ has a single slope $b$ and dimension $n$ then $\rho_{(1,1)}^{(1)}(\FF,\GGG)$ has a single slope $b$ and dimension $mn$.
\end{cor}

\begin{proof}
Let $K,L\in\PPP$ be semisimple, smooth on $\GG_{m,\bar k}-\{1\}$, tame at $0$ and $\infty$ and such that $K_{(1)}\cong\FF$, $L_{(1)}\cong\GGG$. By lemma \ref{vanishing} we have $\rho_{(1,1)}^{(1)}(\FF,\GGG)\cong\R^{-1}\Phi_{(1,1)}$, where $\Phi$ is the vanishing cycles complex for $K\boxtimes L$ and the map $\mu_S:(\GG_{m,\bar k}\times\GG_{m,\bar k})\times_{\GG_{m,\bar k}} S \to S$. By \cite[Expos\'e XIII, Proposition 2.1.4]{deligne1972groupes} we can write this as
$$
\HH^{-1}(\GG_{m,\bar k,(1,1)}^2\times_{\GG_{m,\bar k,(1)}}\bar\eta,K\boxtimes L)
$$
where $\GG_{m,\bar k,(1)}$ (respectively $\GG_{m,\bar k,(1,1)}^2$) is the henselization of $\GG_{m,\bar k}$ at $1$ (resp. the henselization of $\GG_{m,\bar k}^2$ at $(1,1)$), $\bar\eta$ is a geometric point over the generic point of $\GG_{m,\bar k,(1)}$ and the fibre product is taken with respect to the multiplication map $\mu:\GG_{m,\bar k,(1,1)}^2\to\GG_{m,\bar k,(1)}$. Via the translations $t\mapsto t+1$ this is equivalent (with the obvious notation) to
$$
\HH^{-1}(\AAA_{\bar k,(0,0)}^2\times_{\AAA^1_{\bar k,(0)}}\bar\eta,K'\boxtimes L')
$$
where the fibre product is now taken with respect to the map $(x,y)\mapsto (x+1)(y+1)-1=xy+x+y$ and $K'$ and $L'$ are the objects $K$ and $L$ translated by $1$ (so that $K'_{(0)}\cong\FF$ and $L'_{(0)}\cong\GGG$). Using the automorphism $\phi:\AAA^2_{\bar k,(0,0)}\to\AAA^2_{\bar k,(0,0)}$ given by $(x,y)\mapsto (x(y+1),y)$, which fits in a cartesian diagram
$$
\begin{CD}
 \AAA^2_{\bar k,(0,0)} @>\phi >> \AAA^2_{\bar k,(0,0)} \\
@V\sigma+\pi VV  @V\sigma VV \\
 {\AAA^1_{\bar k,(0)}} @> Id >> {\AAA^1_{\bar k,(0)}}
\end{CD}
$$
where $\sigma$ and $\pi$ are the sum and product maps, we get that 
$$
\R^{-1}\Phi_{(1,1)}\cong\HH^{-1}(\AAA_{\bar k,(0,0)}^2\times_{\AAA^1_{\bar k,(0)}}\bar\eta,K'_{x(y+1)}\otimes L'_y)
$$
where the fibre product is now taken with respect to the sum map and $(x,y)$ are the coordinates in $\AAA_{\bar k,(0,0)}^2$. 

 By the bi-exactness of $\rho_{(1,1)}^{(1)}$ we may assume that $\FF$ is a single Jordan block associated to a finite order character $\chi$ of $\bar k^\star$. Since the expression above only depends on the restriction of $K'$ to $\AAA^1_{\bar k,(0)}$ we can take $K'$ to be an indecomposable succesive extension of Kummer objects $\LL_\chi[1]$. Then $\LL_{\chi(x(y+1))}\cong\LL_{\chi(x)}\otimes\LL_{\chi(y+1)}$, so $K'_{x(y+1)}\cong K'_x\otimes\LL_{\chi(y+1)}$ (since $K'_{x(y+1)}$ is a succesive extension of $\LL_{\chi(x(y+1))}[1]$, and it must be indecomposable). But $\LL_{\chi(y+1)}$ is trivial on $\AAA_{\bar k,(0,0)}^2$, so we conclude that 
$$
\R^{-1}\Phi_{(1,1)}\cong\HH^{-1}(\AAA_{\bar k,(0,0)}^2\times_{\AAA^1_{\bar k,(0)}}\bar\eta,K'_{x}\otimes L'_y).
$$

This is just the ``local addivite convolution'' of $K'$ and $L'$ at $0$, as defined in \cite[2.7.2]{laumon1987transformation}. By \cite[Proposition 2.7.2.2]{laumon1987transformation}, we get
$$
\rho_{(1,1)}^{(1)}(\FF,\GGG)\cong\FT_{(0,\infty)}^{-1}((\FT_{(0,\infty)}\FF)\otimes(\FT_{(0,\infty)}\GGG)).
$$
The formula for the dimension and the slope is then a straighforward application of LFTT.
\end{proof}

\begin{cor}
 Let $\FF\in\RR_s^t$ and $\GGG\in\RR_t^t$. Then $\rho_{(s,t)}^{(st)}(\FF,\GGG)\cong\FF\otimes\GGG$.
\end{cor}

\begin{proof}
Suppose that $s=t=1$. By \cite[Proposition 2.5.3.1]{laumon1987transformation} and exactness, the local Fourier transforms of $\FF$ and $\GGG$ are their duals $\hat\FF$ and $\hat\GGG$. Then by the previous corollary
$$
\rho_{(1,1)}^{(1)}(\FF,\GGG)\cong\FT_{(0,\infty)}^{-1}(\hat\FF\otimes\hat\GGG)=\FT_{(0,\infty)}^{-1}(\widehat{\FF\otimes\GGG})=\FF\otimes\GGG
$$
since $\FF\otimes\GGG$ is also tame. The general case follows from corollary \ref{translate}. 
\end{proof}

\begin{rem}\emph{
 If either $\FF$ or $\GGG$ is not tame the previous corollary does not hold even when $s=t=1$, see \cite[3.4.2]{katz1996rls} for some counterexamples.}
\end{rem}

\begin{cor}
 If $K$ and $L$ are everywhere tamely ramified, then $K\ast L$ is everywhere tamely ramified.
\end{cor}

For wild $\FF$ and $\GGG$ we have the following

\begin{prop}\label{bothwild}
 Let $\FF\in\RR_s$, $\GGG\in\RR_t$ of ranks $m$ and $n$ and single slopes $a$ and $b$ respectively. Let $c$ be the Swan conductor of $\FF\otimes \istar\GGG$. Then $\rho_{(s,t)}^{(st)}(\FF,\GGG)$ has dimension $mn(a+b+1)-c$ and Swan conductor $mnab+c$. In particular, if $a<b$ (respectively $a>b$) then $\rho_{(s,t)}^{(st)}(\FF,\GGG)$ has dimension $mn(a+1)$ (resp. $mn(b+1)$) and Swan conductor $mn(a+1)b$ (resp. $mn(b+1)a$).
\end{prop}

\begin{proof}
 Since homotheties do not affect the dimensions or the slopes we may assume by corollary \ref{translate} that $s=t=1$. Let $K,L\in\PPP$ be semisimple, smooth on $\GG_{m,\bar k}-\{1\}$, tame at $0$ and $\infty$ and such that $K_{(1)}\cong\FF$, $L_{(1)}\cong\GGG$, and let $M=K\ast L$. Then $\chi(\Gmk,K)=m(a+1)$ and $\chi(\Gmk,L)=n(b+1)$ by Ogg-Shafarevic, so $\chi(\Gmk,M)=mn(a+1)(b+1)$. By theorems \ref{infinity} and \ref{finite}, $M$ is smooth on $\Gmk-\{1\}$, tame at $0$ and infinity and $M_{(1)}\cong\rho_{(1,1)}^{(1)}(\FF,\GGG)$. We deduce that
\begin{equation}\label{dimplusswan}
mn(a+1)(b+1)=\chi(\Gmk,M)=\dim\rho_{(1,1)}^{(1)}(\FF,\GGG)+\Swan\;\rho_{(1,1)}^{(1)}(\FF,\GGG).
\end{equation}

The generic rank of $\HHH^{-1}(M)$ is $-\chi(\Gmk,K_x\otimes L_{t/x})$ for any $t\in\bar k^\star-\{1\}$. By Ogg-Shafarevic, this is $mn(a+1)+mn(b+1)=mn(a+b+2)$. At $1$, we have $\dim\HHH^{-1}(M)_1-\dim\HHH^0(M)_1=-\chi(\Gmk,K\otimes\istar L)=mn+s$. So $\dim M_{(1)}=mn(a+b+2)-mn-s=mn(a+b+1)-s$, and $\Swan\;M_{(1)}=mnab+c$ by \eqref{dimplusswan}

If $a<b$ (respectively $a>b$) all slopes of $\FF\otimes\istar\GGG$ are equal to $b$ (resp. to $a$) \cite[Lemma 1.3]{katz1988gauss}, so its Swan conductor is $mnb$ (resp. $mna$).
\end{proof}

\begin{rem}\emph{One might ask whether corollary \ref{oneofthemistame} is still valid when both $\FF$ and $\GGG$ are wild since, at least when $a\neq b$, it would give the right dimension and Swan conductor according to the previous proposition. The following example shows that this is not the case.}
\end{rem}

Let $\FF=\LL_{\psi((t-1)^{-2})}$ and $\GGG=\LL_{\psi(-(t-1)^{-2})}$. They are characters of $I_1$ of Swan conductor $2$, and
$$
\FF\otimes\istar\GGG=\LL_{\psi((t-1)^{-2}-(t^{-1}-1)^{-2})}=\LL_{\psi((1+t)/(1-t))}
$$
has Swan conductor $1$, since $\frac{1+t}{1-t}$ has a pole of order $1$ at $t=1$. By proposition \ref{bothwild}, $\rho_{(1,1)}^{(1)}(\FF,\GGG)$ has dimension $4$ and Swan conductor $5$.

Now let $\FF=\LL_{\psi(t^{-2})}$ and $\GGG=\LL_{\psi(-t^{-2})}$ be the same as before, but viewed as representations of $I_0$. Then $K=\LL_{\psi(t^{-2})}[1]$ and $L=\LL_{\psi(-t^{-2})}[1]$ are perverse objects on $\AAA^1_{\bar k}$ which are smooth on $\Gmk$, tame at infinity and such that $K_{(0)}\cong\FF$, $L_{(0)}\cong\GGG$. Since Fourier transform interchanges additive convolution and tensor product, by LFTT we have $(K\ast_+L)_{(0)}\cong\FT_{(0,\infty)}^{-1}((\FT_{(0,\infty)}\FF)\otimes(\FT_{(0,\infty)}\GGG))$, where $K\ast_+L=\R\sigma_!(K\boxtimes L)$ denotes the additive convolution, $\sigma:\AAA^2_{\bar k}\to\AAA^1_{\bar k}$ being the addition map.

Now using an argument similar to the proof of proposition \ref{bothwild}, $(K\ast_+L)_{(0)}$ has dimension $5-c$ and Swan conductor $4+c$, where $c$ is now the Swan conductor of $\FF\otimes\tau_{-1}^\star\GGG$. Since
$$
\FF\otimes\tau_{-1}^\star\GGG=\LL_{\psi(t^{-2}-(-t)^{-2})}=\QQ
$$
is the trivial representation, $c=0$ and therefore $\FT_{(0,\infty)}^{-1}((\FT_{(0,\infty)}\FF)\otimes(\FT_{(0,\infty)}\GGG))$ has dimension $5$ and Swan conductor $4$, and in particular it is not isomorphic to $\rho_{(1,1)}^{(1)}(\FF,\GGG)$.

\section{Tame local monodromy at zero and infinity}

 In the previous two sections we have seen that, for every $K,L\in\PPP$, the local monodromies of $K\ast L$ are completely determined by those of $K$ and $L$, except for the tame part of the monodromies at $0$ and $\infty$. These are the only parts of the local monodromies of a perverse object $K$ that do actually depend on $K$ (as opposed to just on the class of $K$ in $\PPP$). Therefore, in order to give a meaningful result for these monodromies we must work with fixed representatives of the classes in $\PPP$.

 Every class $K\in\PPP$ contains a uniquely determined distinguished element $K_0$: it is the only perverse sheaf isomorphic to $K$ in $\PPP$ that does not have any Kummer sobobject or quotient (it has \emph{property $P$} in the terminology of \cite{katz1996rls}). If $K$ and $L$ do not have Kummer sub-objects or quotients, the distinguished element of $K\ast L\in\PPP$ is the ``middle convolution'' $K\ast_{mid} L$, that is, the image in the category of perverse sheaves on $\Gmk$ of the ``forget supports'' map $K\ast_!L \to K\ast_*L$ \cite[2.6]{katz1996rls}.

If $K$ and $L$ arise from perverse sheaves on $\GG_{m,k}$ which are pure of some weight by extension of scalars to $\bar k$ we have the following result, similar to \cite[Theorem 7.1.4]{katz1988gauss}. For every finite extension $k\subseteq k'$ and every multiplicative character $\chi:k'^\star\to\QQ^\star$ we define the Laurent polynomial $P_{K,\chi}(T)=\sum_{i\in{\mathbb Z}} a_i T^i$, where $a_i$ is the number of unipotent Jordan blocks of size $i$ in the tame part $(\LL_{\bar\chi}\otimes K)_{(\infty)}^t$ of the monodromy of $\LL_{\bar\chi}\otimes K$ at infinity for $i>0$, the number of unipotent Jordan blocks of size $-i$ in the tame part $(\LL_{\bar\chi}\otimes K)_{(0)}^t$ of the monodromy of $\LL_{\bar\chi}\otimes K$ at $0$ for $i<0$, and $a_0$ is such that $P_{K,\chi}(1)$ is the Euler characteristic of $K$ on $\Gmk$.

\begin{prop}
 Let $K=K_0\otimes\bar k$, $L=L_0\otimes\bar k$, where $K_0$ and $L_0$ are perverse sheaves on $\GG_{m,k}$ which are pure of some weight. Suppose that $K$ and $L$ do not have Kummer subobjects or quotients. Then for every finite extension $k\subseteq k'$ and every multiplicative character $\chi:k'^\star\to\QQ^\star$ we have the formula
$$
P_{K\ast_{mid}L,\chi}(T)=P_{K,\chi}(T)P_{L,\chi}(T).
$$
\end{prop}

\begin{proof}
 Taking a geometrically constant twist, we may assume that $K$ and $L$ are pure of weight $0$. By \cite[Chapter 30]{katz2010mellin}, $K\mapsto \omega_\chi(K):=\HH^0(\PP^1_{\bar k},j_{\infty !}\R j_{0\star}(\LL_\chi\otimes K))$ is a fibre functor on the Tannakian category of perverse objects without Kummer subobjects or quotients under the ``middle convolution'' tensor product, so $\omega_\chi(K\ast_{mid} L)\cong\omega_\chi(K)\otimes\omega_\chi(L)$.

Now by \cite[Theorem 16.1]{katz2010mellin}, if $P_{K,\chi}(T)=\sum a_iT^i$ (respectively $P_{L,\chi}(T)=\sum b_j T^j$) $\omega_\chi(K)$ has $a_i$ Frobenius eigenvalues of weight $i$ for every $i\in{\mathbb Z}$ and $\omega_\chi(L)$ has $b_j$ Frobenius eigenvalues of weight $j$ for every $j\in{\mathbb Z}$. So $\omega_\chi(K\ast_{mid}L)$ has $\sum_{i+j=l} a_ib_j$ eigenvalues of weight $l$ for every $l\in{\mathbb Z}$. Therefore
\begin{align*}
P_{K\ast_{mid}L,\chi}(T)&=\sum_l\left(\sum_{i+j=l} a_ib_j\right)T^l=
\\
&=\left(\sum_i a_iT^i\right)\left(\sum_j b_jT^j\right)=P_{K,\chi}(T)P_{L,\chi}(T).
\end{align*}
\end{proof}

We conjecture that the formula is still true for arbitrary semisimple $K$ and $L$. See \cite[7.5]{katz1988gauss} for O. Gabber's proof in the case where $K,L\in{\mathcal C}$.

\bibliographystyle{amsalpha}
\bibliography{convolution}

\providecommand{\bysame}{\leavevmode\hbox to3em{\hrulefill}\thinspace}
\providecommand{\MR}{\relax\ifhmode\unskip\space\fi MR }
\providecommand{\MRhref}[2]{%
  \href{http://www.ams.org/mathscinet-getitem?mr=#1}{#2}
}
\providecommand{\href}[2]{#2}
\begin{thebibliography}{Kat88b}

\bibitem[Bry86]{brylinski1986transformations}
Jean-Luc Brylinski, \emph{Transformations canoniques, dualit\'e projective,
  th\'eorie de {L}efschetz, transformations de {F}ourier et sommes
  trigonom\'etriques}, Ast\'erisque (1986), no.~140-141, 3--134,
  G{\'e}om{\'e}trie et analyse microlocales.

\bibitem[Del77]{deligne569application}
P.~Deligne, \emph{Application de la formule des traces aux sommes
  trigonom{\'e}triques}, iv+312pp, Cohomologie \'Etale, S{\'e}minaire de
  G{\'e}om{\'e}trie Alg{\'e}brique du Bois-Marie SGA 4\textonehalf.

\bibitem[DK73]{deligne1972groupes}
P.~Deligne and N.M. Katz, \emph{Groupes de monodromie en g\'eom\'etrie
  alg\'ebrique}, Lecture Notes in Mathematics, Vol. 340, Springer-Verlag,
  Berlin, 1973, S{\'e}minaire de G{\'e}om{\'e}trie Alg{\'e}brique du Bois-Marie
  1967--1969 (SGA 7 II).

\bibitem[GL96]{gabber1996faisceaux}
Ofer Gabber and Fran{\c{c}}ois Loeser, \emph{Faisceaux pervers {$l$}-adiques
  sur un tore}, Duke Math. J. \textbf{83} (1996), no.~3, 501--606.

\bibitem[Kat86]{katz1986local}
Nicholas~M. Katz, \emph{Local-to-global extensions of representations of
  fundamental groups}, Ann. Inst. Fourier (Grenoble) \textbf{36} (1986), no.~4,
  69--106.

\bibitem[Kat88a]{katz1988gauss}
\bysame, \emph{Gauss sums, {K}loosterman sums, and monodromy groups}, Annals of
  Mathematics Studies, vol. 116, Princeton University Press, Princeton, NJ,
  1988.

\bibitem[Kat88b]{katz1988travaux}
\bysame, \emph{Travaux de {L}aumon}, Ast\'erisque (1988), no.~161-162, Exp.\
  No.\ 691, 4, 105--132 (1989), S{\'e}minaire Bourbaki, Vol. 1987/88.

\bibitem[Kat90]{katz1990esa}
\bysame, \emph{Exponential sums and differential equations}, Annals of
  Mathematics Studies, vol. 124, Princeton University Press, Princeton, NJ,
  1990.

\bibitem[Kat96]{katz1996rls}
\bysame, \emph{Rigid local systems}, Annals of Mathematics Studies, vol. 139,
  Princeton University Press, Princeton, NJ, 1996.

\bibitem[Kat11]{katz2010mellin}
\bysame, \emph{{S}ato-{T}ate theorems for finite-field {M}ellin transforms},
  Preprint, available at \url{http://math.princeton.edu/~nmk} (2011).

\bibitem[KL85]{katz1985transformation}
Nicholas~M. Katz and G{\'e}rard Laumon, \emph{Transformation de {F}ourier et
  majoration de sommes exponentielles}, Inst. Hautes \'Etudes Sci. Publ. Math.
  (1985), no.~62, 361--418.

\bibitem[Lau81]{laumonsemi}
G.~Laumon, \emph{Semi-continuit\'e du conducteur de {S}wan (d'apr\`es {P}.
  {D}eligne)}, The {E}uler-{P}oincar\'e characteristic ({F}rench),
  Ast\'erisque, vol.~83, Soc. Math. France, Paris, 1981, pp.~173--219.

\bibitem[Lau87]{laumon1987transformation}
\bysame, \emph{Transformation de {F}ourier, constantes d'\'equations
  fonctionnelles et conjecture de {W}eil}, Inst. Hautes \'Etudes Sci. Publ.
  Math. (1987), no.~65, 131--210.

\bibitem[RL10]{rl2010}
A.~Rojas-Le\'on, \emph{Rationality of trace and norm {L}-functions}, Preprint,
  \url{http://arxiv.org/abs/1007.5324} (2010).

\end{thebibliography}

\end{document}